\documentclass[reqno,11pt]{amsart}

\usepackage{amsfonts, amssymb, verbatim}
\usepackage{amssymb,epsfig,amscd, verbatim}
\usepackage{amsthm}
\usepackage{amsmath}
\usepackage{accents}
\usepackage{graphicx}
\usepackage{latexsym}

\usepackage[all]{xy}

\newtheorem{theorem}{Theorem}
\newtheorem{lemma}[theorem]{Lemma}
\newtheorem{proposition}[theorem]{Proposition}
\newtheorem{corollary}[theorem]{Corollary}

\theoremstyle{definition}
\newtheorem{definition}[theorem]{Definition}
\newtheorem{example}[theorem]{Example}

\numberwithin{equation}{section} \numberwithin{theorem}{section}

\theoremstyle{remark}
\newtheorem{remark}[theorem]{Remark}

\def\z{\,_{\dot z}\,}
\def\w{\,_{\dot w}\,}
\def\zw{\,_{\dot {z+w}}\,}

\def\vac{|0\rangle}                            %% vacuum vector
\def\g{{\mathfrak{g}}}      % Lie algebra g
\newcommand\lbb[1]{\label{#1}}
\def\tt{\otimes}                               %% tensor product
\def\<{\langle}
\def\>{\rangle}

\def\dsum{\displaystyle\sum}
\def\st{\; | \;}                               %%  such that
\def\vac{\mathbf{1}}                            %% vacuum vector
\def\one{{\mathbf{1}}}

\newcommand{\kk}{\mathbf{k}}

\newcommand{\NN}{\mathbb{N}}

\newcommand{\ZZ}{\mathbb{Z}}

\def\de{\delta}

\def\g{{\mathfrak{g}}}      % Lie algebra g
\def\z{\,_{\dot z}\,}
\def\x{\,_{\dot x}\,}
\def\y{\,_{\dot y}\,}
\def\w{\,_{\dot w}\,}
\def\r{\,_{\dot r}\,}
\def\sss{\,_{\dot s}\,}
\def\t{\,_{\dot t}\,}

\def\mz{\,_{\dot {-z\ }}\,}
\def\zx{\,_{\dot {z+x}}\,}
\def\zy{\,_{\dot {z+y}}\,}
\def\yz{\,_{\dot {y+z}}\,}
\def\xz{\,_{\dot {x+z}}\,}
\def\xxy{\,_{\dot {x+y}}\,}
\def\yx{\,_{\dot {y+x}}\,}
\def\xr{\,_{\dot {x+r}}\,}
\def\xmy{\,_{\dot{x-y}}\,}

\def\myx{\,_{\dot{-y+x}}\,}
\def\mxr{\,_{\dot{-x+r}}\,}
\def\xmz{\,_{\dot{x-z}}\,}
\def\zmy{\,_{\dot{z-y}}\,}
\def\myz{\,_{\dot{-y+z}}\,}
\def\mzx{\,_{\dot{-z+x}}\,}

\def\zmx{\,_{\dot{z-x}}\,}
\def\mxz{\,_{\dot{-x+z}}\,}
\def\mzx{\,_{\dot{-z+x}}\,}
\def\xzy{\,_{\dot{x+z+y}}\,}

\def\ozz{\underaccent{  z}{\otimes}}
\def\oz{\underaccent{\ \  z}{\otimes}}
\def\omz{\underaccent{\ \  -z}{\otimes}}

\def\ox{\underaccent{\ \  x}{\otimes}}
\def\oy{\underaccent{\ \  y}{\otimes}}
\def\oV{\underaccent{\ \  V}{\otimes}}
\def\oxy{\underaccent{\ \  x+y}{\otimes}}
\def\oxmy{\underaccent{\ \  x-y}{\otimes}}
\def\ozy{\underaccent{\ \  z+y}{\otimes}}
\def\ozmy{\underaccent{\ \  z-y}{\otimes}}
\def\oy{\underaccent{\ \  y}{\otimes}}
\def\oV{\underaccent{\ \  V}{\otimes}}

\def\s{\:\!}

\def\noi{\noindent}

\def\rz{\overset{{\scalebox{.64}{$r$}}\hskip .01cm}{\z}}

\def\Mz{\overset{{\scalebox{.64}{$M$}}\hskip .01cm}{\z}}
\def\Mx{\overset{{\scalebox{.64}{$M$}}\hskip .01cm}{\x}}
\def\My{\overset{{\scalebox{.64}{$M$}}\hskip .01cm}{\y}}
\def\Mt{\overset{{\scalebox{.64}{$M$}}\hskip .01cm}{\t}}

\def\Nz{\overset{{\scalebox{.64}{$N$}}\hskip .01cm}{\z}}
\def\Nx{\overset{{\scalebox{.64}{$N$}}\hskip .01cm}{\x}}
\def\Ny{\overset{{\scalebox{.64}{$N$}}\hskip .01cm}{\y}}
\def\Nt{\overset{{\scalebox{.64}{$N$}}\hskip .01cm}{\t}}

\def\Wx{\overset{{\scalebox{.64}{$W$}}\hskip .01cm}{\x}}
\def\Wy{\overset{{\scalebox{.64}{$W$}}\hskip .01cm}{\y}}
\def\Wt{\overset{{\scalebox{.64}{$W$}}\hskip .01cm}{\t}}

\def\Nzy{\overset{{\scalebox{.64}{$N$}}\hskip .01cm}{\zy}}

\def\Mz{\overset{{\scalebox{.64}{$M$}}\hskip .01cm}{\z}}
\def\Mz{\overset{{\scalebox{.64}{$M$}}\hskip .01cm}{\z}}

\def\Mzx{\overset{{\scalebox{.64}{$M$}}\hskip .01cm}{\zx}}

\def\Mx{\overset{{\scalebox{.64}{$M$}}\hskip .01cm}{\x}}

\def\Mzy{\overset{{\scalebox{.64}{$M$}}\hskip .01cm}{\zy}}
\def\Myz{\overset{{\scalebox{.64}{$M$}}\hskip .01cm}{\yz}}

\def\MNW{ \binom{W}{M\, ,\, N}}

\def\ddx{\frac{d}{dx}\,}

\def\ddz{\frac{d}{dz}\,}
\def\dxyz{\delta(x-y,z)}
\def\dmyxz{\delta(-y+x,z)}
\def\dxzy{\delta(x-z,y)}

\def\VA{(V,\z,\vac,d)}
\def\v2a{(V,\z,\vac,d)}
\begin{document}

\title{Tensor product of modules
over a vertex algebra}

\author[Jos\'e I. Liberati]{Jos\'e I. Liberati$^*$}
\thanks {\textit{$^{*}$Ciem - CONICET, Medina Allende y
Haya de la Torre, Ciudad Universitaria, (5000) C\'ordoba -
Argentina. \hfill \break \indent e-mail: joseliberati@gmail.com}}
\address{{\textit{Ciem - CONICET, Medina Allende y
Haya de la Torre, Ciudad Universitaria, (5000) C\'ordoba -
Argentina. \hfill \break \indent e-mail: joseliberati@gmail.com}}}

%\date{version final 21 sept, 2016}


\maketitle

\begin{abstract}
We found a necessary and sufficient condition for the existence of
the tensor product of modules over a vertex algebra. We defined
the notion of vertex bilinear map and we provide two  algebraic
construction of the tensor product, where one of them is of ring
theoretical type. We show the relation between the tensor product
and the vertex homomorphisms. We prove the commutativity of the
tensor product. We also prove the associativity of the tensor
product of modules under certain necessary and sufficient
condition. Finally, we show certain functorial properties of
 the vertex homomorphims and the tensor product.
\end{abstract}
%\maketitle
%\tableofcontents

%%%%%%%%%%%%%%%%%%%%%%%%%%%%%%%%%%%%%%%%%%%%%%%%%%%%%%%%%%%%%%%%%%%%%%%%%

%%%%%%%%%%%%%%%%%%%%%%%%%%%%%%%%%%%%%%%%
\section{Introduction}\lbb{intro}
%%%%%%%%%%%%%%%%%%%%%%%%%%%%%%%%%%%%%%%%

This is the first of a series of papers (the second one is \cite{L1}) trying to
extend certain restricted definitions and constructions developed
for vertex operator algebras to the general framework of vertex
algebras without assuming any grading condition neither on the
vertex algebra nor on the modules involved, and making a strong
emphasis on the commutative associative algebra point of view
instead of the  Lie theoretical point of view. These series of
works help to eliminate the Jacobi identity from the theory of
vertex algebras and clean it by considering the correct point of
view based on the theory of commutative associative algebras.

The study of tensor product theory for representations of a chiral
algebra (or vertex operator algebra),  was initiated by physicists
 (\cite{MS} and references there in). After that, Gaberdiel \cite{G} pointed
out that Borcherds suggested a ring like construction, and he
developed it on a physical level of rigor. Simultaneously, Huang
and Lepowsky \cite{HL1}-\cite{HL5} and \cite{H}, gave the first mathematically
rigorous notion and construction of the tensor product of certain
graded modules over certain vertex operator algebras, by using analytical
methods. In  \cite{Li1}, Li gave a formal variable approach to the
tensor product theory, following classical Lie algebra theory and
restricting the definitions and constructions to certain graded
modules over a rational vertex operator algebra.

Many mathematicians believe that there should be an algebraic and
ring type construction of the tensor product of modules over a
vertex algebra. It has been a necessary and missing part of the
theory  in the last 30 years. The present work try to fill this
gap.

In order to have a purely algebraic definition and construction,
we loose the analytical and geometric interpretation, but we
simplify the theory.

Our work is strongly influenced by the works of Haisheng Li,
specially his thesis. Most of the ideas and tricks used in this work are taken
from his papers, but applied in a different context. The ideas and
 constructions are very simples and naturals.

In section 2, we   present the basic definitions and notations.
In section 3, we found a necessary and sufficient condition for the existence of
the tensor product of modules over a vertex algebra, that we called the kernel intertwining
operator full equality condition, and we present the first construction.

Since we have  a
different point of view of very well known notions, in section 4 we define
the notion of vertex bilinear map that replaces the notion of
intertwining operator. This allow us to redefine the notion of tensor product
  with a ring theoretical interpretation (Definition \ref{tensor2}), and we
  provide the second  algebraic
construction of the tensor product, which is of ring
theoretical type.

In  section 5, we introduce the notion of (right) vertex
homomorphism, producing  the "$Hom$" functor for modules over a
vertex algebra, obtaining (from our point of view) an analog of
the $Hom$ functor for modules over an associative commutative
algebra, and we denote it as $Vhom$. We prefer to think about
it in this way, instead of the
$Hom$ functor in Lie theory. In the first part of this section, we
follow \cite{Li1}, \cite{Li2} and \cite{DLM}.
The most important functors in homological algebra are $Hom$,
tensor product and functors derived from them. In the case of
vertex algebras, we should consider $Vhom$ (and later $Vhom^r$).
We shall show that there is an intimate relationship between
$Vhom$ and tensor product, they form an adjoint pair of functors (see
Theorem \ref{TT}, Theorem \ref{TT22} and Corollary \ref{TTTTT}).

In section 6, we prove the commutativity of the tensor product.
 Then we prove the associativity of the tensor product under
certain (algebraic and natural) necessary and sufficient
condition. In order to prove the associativity, we shall
use the universal property of the tensor
product instead of the explicit construction, following in part
the ideas in \cite{DLM} but using algebraic methods instead of
analytic ones. We  simplify Huang's convergence
assumptions.

In section 7, we  introduce the notions of vertex projective, vertex
injective and vertex flat $V$-modules. Roughly speaking,
homological algebra is concerned with the question of how much
modules differ from being projective, injective and flat.
We do not try to develop homological algebra theory for vertex
algebras in this work. This is just an introduction or motivation
to the reader to try to develop it from the point of view of a
vertex algebra as a generalization of an associative commutative
algebra with unit. The point of view and the results of this work could be the
starting point of the development of homological algebra for vertex
algebras.

The results and constructions of this work could be used to
simplify the construction of the tensor product by Kazhdan and
Lusztig in \cite{KL1}-\cite{KL5}, by using the relation with the tensor
product of modules over the vertex algebras associated to the
affine Kac-Moody algebras founded by Zhang \cite{Zha} and references
 there in. It is interesting to see possible consequences on the tensor
product of modules over quantum groups.

In \cite{L2}, we translate the ideas of this work to the construction of the
tensor product of conformal modules over a Lie conformal algebra,
and in \cite{L3} the construction is done for $H$-pseudo algebras,
answering an open question suggested by V. Kac during a graduate
course at MIT in 1997.

After reading this work, there is no doubt  that this is the
correct point of view and that the notation that we introduced
(many years ago) makes it an elegant work. On may say that many
results of this paper follow by standard arguments and therefore
it is not a valuable work, but the {\it key part} of this paper
(and \cite{L1}) is that we found the correct approach, and the
{\it consequence} of this is that the main theorems follow by
standard arguments.

Finally, for me it is important to point out that in 2011
we found the kernel condition and the first construction
of the tensor product, but due to personal problems we abandoned
this paper until January 2014, when  we found the notion of vertex
bilinear map and the second construction. We typed this work,
including the commutativity property of the tensor product and the
relation with $Vhom$. It was basically finished  in 2014, except
for the associativity and the final section. In May 2016, we resume this work
by checking the typos, and   we decided to complete it with the proof of
associativity and the study of the functorial properties of $Vhom$
and the tensor product written in the last section, producing this
final version.

Unless otherwise specified, all vector spaces, linear maps and
tensor products are considered over an algebraically closed field
$\kk$ of characteristic 0.

%%%%%%%%%%%%%%%%%%%%%%%%%%
\section{Definitions and notation}\lbb{def }
%%%%%%%%%%%%%%%%%%%%%%%%%%

\

In order to make a self-contained paper, in  this section we
present the notion of  vertex algebra and their modules,
intertwining operators and tensor product. Our presentation and
notation differ from the usual one because we want to emphasize
the point of view that vertex algebras are analog to commutative
associative algebras with unit.

\

Throughout this work, we define $(x + y)^n$ for $n\in\ZZ$ (in
particular, for $n < 0$) to be the formal series
\begin{align*}
(x+y)^n=\sum_{k \in \mathbb{Z}_+} \binom{n}{k} x^{n-k}y^k,
\end{align*}
where $\binom{n}{k}=\frac{n(n-1)...(n-k+1)}{k!}$.

\

\definition \label{def VA}
A {\it vertex algebra} is a quadruple $\VA$ that consists of a
$\kk$-vector space $V$ equipped with a linear map
 \begin{align}\label{z1}
 \z \, :V &\otimes V \longrightarrow V((z))\\
 a &\otimes b \, \, \,  \longmapsto a\z b, \nonumber
 \end{align}
 a distinguished vector  $\vac$ and $d\in\,$End$(V)$
satisfying the following axioms ($a,b,c \in  V):$\\
\vskip.1cm \noindent$\bullet$ {\it Unit}:
\begin{align*}
 \vac \z a= a \quad  \textrm{ and }& \quad a \z
\vac=e^{zd}a;
\end{align*}

\noindent $\bullet$ {\it Translation - Derivation}: \vskip .1cm
\begin{align}\label{D}
(da)\z b= \ddz (a\z b), \quad d(a\z b)=(da)\z b + a\z (db);
\end{align}
\vskip .3cm

\noindent$\bullet$ {\it Commutativity}:
\begin{align} \label{skew}
a \z b= e^{zd} (b\mz a);
\end{align}
\vskip .3cm

\noindent$\bullet$  {\it Associativity}: For any $a,b,c\in V$,
there exists $l\in \NN$ such that

\begin{align} \label{aociator}
(z+w)^l\, \ (a \z b) \w c = (z+w)^l\, \  a\zw (b\w c) .
\end{align}
\vskip.3cm

\

\begin{remark}\label{r} (a) Observe that the standard notation $Y(a,z)\, b$ for the
$z$-product in (\ref{z1}) has been changed. We adopted this
notation following the practical idea  of the  $\lambda$-bracket
in the notion of Lie conformal algebra (also called vertex Lie
algebra in the literature), see \cite{K}. In fact, we have been
using this notation since 2011 (see the undergraduate thesis of my
student in \cite{O}).

\noindent(b) If we write $a\z b$ in terms of its coefficients we
obtain the definition of the {\it $n$-products} $a\,_{n} b$ given
by
\begin{equation*}
    a\z b= \sum_{n\in\ZZ} (a\,_{n} b) \, z^{-n-1}.
\end{equation*}

\noindent(c) Sometimes we shall use the following standard
notation for the {\it vertex operators}:
\begin{equation*}
\ a\, _{\dot z} \, =\dsum_{n\in
      \mathbb{Z}} \; a_{n} \; z^{-n-1}
\end{equation*}
where $a_{n}\in\;$End$(V)$ is the endomorphism determined by the
$n$-product. Since the image of the map $\z$ is in $V((z))$  we
have that for all $a,b \in V $ there exists $N_{a,b}\in\NN$ such
that
\begin{equation*}
a \z b=\dsum_{n < N_{a,b}} (a_n b) z^{-n-1}.
\end{equation*}
That is, for any $a,b \, \in V$,  there exists $N_{a,b} \geq 0$
such that
\begin{align*}
a_nb=0 \quad \forall \, \, n   \geq N_{a,b}.
\end{align*}

\noi (d) The associativity axiom can be replaced by the same
identity (\ref{aociator}), with the integer $l$ depending only on
$a$ and $c$.
\end{remark}

 The commutativity axiom is known in the literature as
skew-symmetry (see \cite{LL, K}), but for us it really corresponds
to commutativity. We want to emphasize the point of view of a
vertex algebra as a generalization of an associative commutative
algebra with unit (and a derivation), as in \cite{BK, Li3, Bo2},
and having in mind the usual (trivial) example of holomorphic
vertex algebra given by any associative commutative algebra $A$
with unit $\one$ and a derivation $d$, with the $\z$-product
defined by $a\z b:=(e^{zd} a)\cdot b$. In the special case of
$d=0$, we simply have an associative commutative algebra with unit
and $a\z b:= a \cdot b$, and with the definition presented before,
it is immediate to see, without any computation, that this is a
vertex algebra, and this is the reason for us to call the usual
skew-symmetry axiom as commutativity.

An equivalent definition can be obtained by replacing the
associativity axiom by the {\it associator formula} (which is
equivalent to what is  known in the literature as the iterate
formula (see \cite{LL}, p.54-55) or the $n$-product identity (see
\cite{BK})):
\begin{align} \label{associator}
(a \z b) \y c- a\zy (b\y c)=  b\y (a\yz c - a\zy c)\qquad \forall\
a,b,c\in V.
\end{align}
Observe that in the last term of the associator formula we can not
use linearity to write it as a difference of $\ b\,_{\dot
w}\big(\,a\,_{\dot {w+z}}\,c\big)$ and $\ b\,_{\dot
w}\big(\,a\,_{\dot {z+w}}\,c\big)$, because neither of these
expressions in general exist (see \cite{LL}, p.55). This
alternative definition of vertex algebra using the (iterate
formula or) the associator formula is essentially the original
definition given by Borcherds \cite{Bo1}, but in our case it is
written using the generating series in $z$ instead of the
$n$-products.

Another equivalent definition can be obtained by replacing the
associativity axiom (and commutativity) by the {\it Jacobi
identity}:
\begin{align}\label{jacobi}
\dxyz \, a\x b\y- \  \dmyxz\,  b\y a\x= \dxzy\, (a \z b)\y  \quad
\forall \, a,b \in V,
\end{align}
where $\delta(x,y):=\sum_{n \in \mathbb{Z}} x^n y^{-n-1}$ is the
{\it delta function}.

It is well known that taking Res$_{\,x}$ (that is the coefficient in
$x^{-1}$) of the Jacobi identity we obtain the Associator formula
and, using Remark \ref{r} (c), by multiplying the Associator
formula (\ref{associator}) by $(y+z)^{N_{a,c}}$ we get
Associativity. The rest of the proof of the equivalence between
these definitions is more complicated (see \cite{LL}). It is well
known (see \cite{LL}) that in the three equivalent definitions
that we presented, some of the axioms can  be obtained from the
others, but we prefer to make
 emphasis    on the properties of $d$ and the explicit formula for the multiplication
 by the unit.

Another useful formula, called the {\it Commutator formula}, is
obtained by taking Res$_{\,z}$ in the Jacobi identity:
\begin{equation}\label{commutator}
a\x (b\y c)- b\y (a\x c)=(a\xmy b\, - \, a \myx b)\y c.
\end{equation}

\vskip .3cm

\noi from which it is easy to deduce what is called {\it weak
commutativity} or {\it locality} in the literature: for all
$a,b\in V$, there exists $k\in\NN$ such that
\vskip .01cm
\begin{equation}\label{locality}
    (x-y)^k \, a\x (b\y c)= (x-y)^k \,b\y (a\x c) \quad  \hbox{for all }c\in V.
\end{equation}

\

\begin{definition} A {\it module} over a vertex algebra $V$ is
a $\kk$-vector space $M$ equipped with an endomorphism $d$ of $M$
and a linear map

\begin{equation*}%\label{z-product}
 V\otimes M \longrightarrow M((z)), \qquad \quad (a\,,u)\
 \longmapsto \ a \Mz u=\dsum_{n\in
      \mathbb{Z}} \ (a\,_{n} u)\ z^{-n-1}
%    \begin{array}{c}
%      V\otimes V \longrightarrow V((z)) \qquad \qquad \qquad\qquad \ \\
%      (a,b)\ \longmapsto \ a\, _{\dot z} \, b=\dsum_{n\in
%      \mathbb{Z}} \ a_{(n)} b\ z^{-n-1}
%    \end{array}
\end{equation*}
satisfying the following axioms ($a,b \in  V$ and $u\in M$):
\\
\vskip.1cm \noindent$\bullet$ {\it Unit}:
\begin{align*}
 \vac \Mz u= u;
\end{align*}
$\bullet$ {\it Translation - Derivation}:
\begin{align}\label{MD}
(da)\Mz u= \ddz (a\Mz u), \quad d(a\Mz u)=(da)\Mz u + a\Mz (d\,u);
\end{align}
\vskip.2cm \noindent $\bullet$  {\it Associativity}:  For any
$a,b\in V$ and $u\in M$, there exists $l\in\NN$ such that

\begin{equation}\label{456}
(z+x)^l\, \ (\,a\, _{\dot z} \, b\, )\Mx u\ =(z+x)^l\ \,a \Mzx (\,
b\Mx u\,)\ .
\end{equation}
\end{definition}

\vskip .1cm

\begin{remark}\label{module}
(a)  Sometimes, if everything is clear, we shall use $a\z u$
instead of $a \Mz u$. We shall also use the notation $a\Mz$.
Obviously, $V$ is a module over $V$.

\noindent (b) We follow  \cite{L1,BK}, in the definition of
module, because in order to talk about intertwining operators we
need to work with this category
 of modules, namely, with this $\kk[d]$-module structure
 (similar to the situation of  Lie conformal algebras \cite{K}).
  We denote it by $d$, but it should be $d_M$.

\noindent (c) Similarly to the definition of vertex algebra, the
associativity axiom could be replaced by the associator formula:

\begin{align}\label{assoc}
(\,a\, _{\dot z} \, b\, )\y v\ &=\ \,a\,_{\dot {z+y}}\,(\, b\y
v\,)\ + \ b\,_{\dot y}\big(\,a\yz v\,-\,a\zy v\,\big)\nonumber\\
&=\ \,a\,_{\dot {z+y}}\,(\, b\y v\,)\ - \mathrm{Res}_x\
\de(-y+x,z) \ b\y (a\x v).
\end{align}

\vskip .3cm

\noi or by the Jacobi identity. As in the vertex algebra case, the
associativity axiom can be replaced by the same identity
(\ref{456}), with the integer $l$ depending only on $a$ and $u$.
Any $V$-module $M$ satisfies  the {\it Commutator formula for
modules} given by (for any $a,b\in V$ and $u\in M$):
\begin{equation}\label{commutator-m}
a\x (b\y u)- b\y (a\x u)=(a\xmy b\, - \, a \myx b)\y u.
\end{equation}
from which it is easy to deduce what is called {\it weak
commutativity} or {\it locality} for a $V$-module $M$: for all
$a,b\in V$, there exists $k\in\NN$ such that
\begin{equation}\label{locality-m}
    (x-y)^k \, a\x (b\y u)= (x-y)^k \,b\y (a\x u) \quad  \hbox{for all }u\in M.
\end{equation}
It is well known that commutator formula for modules can not be
used to replace associativity axiom (commutativity of left
multiplication operators does not imply module structure over an
associative commutative algebra).

\noi (d) For $a\in V$ and $u\in M$, we take $N_{a,u}$ the minimum
positive integer such that
\begin{equation}\label{nnnnn}
z^{N_{a,u}} \,
(a\Mz u)\in M[[z]].
\end{equation}

\noi (e) If $V$ is an associative commutative algebra with unit
and $d=0$, then (by translation axiom) we have that $\ddz (a\Mz
u)=0$. Hence $a\Mz u$ is independent of $z$ and a module over the
vertex algebra $V$ corresponds to a module over the associative
commutative algebra with unit $V$ together with $d=d_M\in
End\s_V(M)$. If $d_M=0$, then we recover the usual notion of
module over an associative commutative algebra with unit.

 \end{remark}

\

Let $M$ and $N$ be $V$-$\,$modules, a {\it $V$-$\,$homomorphism} or a {\it
homomorphism of $V$-$\,$modules} from $M$ to $N$ is a linear map
$\varphi:M\to N$ such that for $a\in V$ and $u\in M$
$$
\varphi(a\z u)= a\z \varphi(u) \ \ \mathrm{ and }\ \ \varphi(d u)=
d \varphi(u) .
$$
The space of homomorphisms of $V$-$\,$modules is denoted $Hom_V(M ,
N)$.

\

We shall need the following well known results:

\begin{proposition} \label{lie}(a) Let $V$ be a vertex algebra. Then the quotient space
$$
\g(V):=\kk [t,t^{-1}]\tt V /\bigg(\frac{d}{dt}\tt 1 +1\tt
d\bigg)(\kk [t,t^{-1}]\tt V ),
$$
is a Lie algebra associated to $V$ with Lie
bracket:
\begin{equation*}
[\overline{t^m\tt a},\overline{t^n\tt b}]=\sum_{i=0}^\infty
\binom{m}{i} \overline{t^{m+n-i}\tt a_i b}.
\end{equation*}
for any $a,b\in V$.

\noindent(b) Any $V$-$\,$module is a $\g(V)$-module, where $t^m\tt a$
acts as $a_m$.

\noindent(c) For any  $V$-$\,$module $M$, define $\widehat M=\kk
[t,t^{-1}]\otimes M$. Then $\widehat M$ is a $\g(V)$-module with the
action  given by:
\begin{equation*}%\label{action afin}
    (t^m\tt a)(t^n\tt v)=\sum_{i=0}^\infty \binom{m}{i}
    t^{m+n-i}\otimes a_i v.
\end{equation*}
for any $a\in V$ and $v\in M$.

\end{proposition}

As an abuse of notation, we remove the bar in the class elements
in $\g(V)$.

 \ 

Now, we introduce the notions of intertwining operators and tensor
product of modules, and both notions will be redefined from a ring
theoretical point of view in Section \ref{secondoo}.

\vskip .3cm

\begin{definition}\label{intertwining}
Let   $M,N$ and $W$ be three $V$-$\,$modules. An   {\it intertwining
operator} of type $\MNW$ is a $\kk$-bilinear map
\begin{align*}%\label{I}
 I_z \, :M \times N &\longrightarrow W((z))\\
 (u , v) \, \, \,  &\longmapsto I_z(u, v)=\sum_{n\in \ZZ} \,
  I_{_{(n)}}(u, v) \, z^{-n-1}, \nonumber
 \end{align*}
 satisfying the following conditions:
 \vskip.2cm
\noindent $\bullet$ {\it Translation - Derivation}:
\begin{align}\label{ID}
I_z(d\, u , v)= \ddz I_z(u , v),\quad \mathrm{and }\ \ d\,(I_z( u
, v))=I_z(d\, u , v)+I_z( u , d\, v),
\end{align}
\vskip.2cm \noindent $\bullet$  {\it Jacobi identity}: For all $ a
\in V, u\in M$ and $v\in N$

\begin{align}
\dxyz \, a\x I_y(u , v) - \  \dmyxz\,  I_y(u, a\x v)= \dxzy\,
I_y(a \z u , v)  \, .
\label{Ijacobi}
\end{align}
\end{definition}

\

\definition\label{tensor}
Let $M$ and $N$ be two $V$-$\,$modules. A pair $(M\oV N, F_z)$, which
consists of a $V$-$\,$module $M\oV N$ and an intertwining operator
$F_z$ of type $\binom{_{M\oV N}}{^M\, , \, ^N}$,  is called a {\it
tensor product} for the ordered pair $(M,N)$ if the following
universal property holds: For any $V$-$\,$module $W$ and any
intertwining operator $I_z$ of type $\MNW$ there exists a unique
$V$-$\,$homomorphism $\varphi$ from $M\oV N$ to $W$ such that
$I_z=\varphi \circ F_z$, where $\varphi$ is extended canonically
to a linear map from $(M\oV N)((z))$ to $W((z))$.

\

Just as in the classical algebra theory, it follows from the
universal property that if there exists a tensor product for the
ordered pair $(M,N)$, then it is unique up to a $V$-$\,$module
isomorphism. Namely, the pair $(W,G_z)$ is another tensor product
if and only if there exists a $V$-$\,$module isomorphism $\phi:M\oV
N\to W$ such that $G_z=\phi\circ F_z$.

Using the Jacobi identity (\ref{Ijacobi}), it is immediate that
any intertwining operator $I_z\in \MNW$ satisfies the associator
formula (for $a\in V$, $u\in M$ and $v\in N$)
\begin{equation}\label{Iassoc}
a\zx I_x(u,v)=I_x(a\z u,v)+I_x(u, a\zx v - a\xz v),
\end{equation}
the iterated formula
\begin{equation}\label{Ittt}
I_z(a\x u,v)=\hbox{Res}\s_y \big[\delta (y-z,x)a\y
I_z(u,v)-\delta(-z+y,x)I_z(u,a\y v)\big],
\end{equation}
and the commutator formula
\begin{equation}\label{Icommut}
    a\x I_y(u,v)-I_y(u,a\x v)=I_y(a\xmy u -a\myx u,v).
\end{equation}

\

%%%%%%%%%%%%%%%%%%%%%%%%%%
\section{First construction of the tensor product}\lbb{first}
%%%%%%%%%%%%%%%%%%%%%%%%%%

\   

First of all, we need some definitions in order to find necessary
and sufficient conditions for the existence of the tensor product.

\

\definition
(a) Let $M$ and $N$ be two $V$-$\,$modules.  We define the {\it kernel
of the intertwining operators from the pair $(M,N)$} as follows
\begin{align*}
    Ker\, \binom{\cdot}{M,N} :=\bigg\{ &(u,v)\in M\times N  \st \exists\   l_{u,v}\in \NN
 \textrm{ such that  } I_{_{(n)}}(u,v)=0 \nonumber\\&
 \textrm{ for all }n\geq l_{u,v}, \textrm{ for all }
 V\textrm{-modules $W$, and for all }I_z\in
 {\vcenter{\hbox{$\genfrac{(}{)}{0pt}{1}{W}{M,N}$}}}
   \bigg\}\nonumber
\end{align*}
(b) We say that the pair $(M,N)$ satisfies the {\it kernel
intertwining operator full equality condition} if
\begin{align*}%\label{K-cond}
    M\times N=Ker\, \binom{\cdot}{M,N}.
\end{align*}

\

\begin{proposition} (Necessary condition) If the tensor product
$(M\oV N, F_z)$ exists, then the pair $(M,N)$ satisfies the kernel
intertwining operator full equality condition.
\end{proposition}

\begin{proof} Fix $(u,v)\in M\times N$ and using that $F_z(u,v)\in (M\oV N)((z))$,
 we have that there exists $N\in \NN$ (depending on $u$ and $v$)
 such that $F_{_{(n)}}(u,v)=0$ for all $n\geq N$.
 Since for any $V$-$\,$module $W$ and any
intertwining operator $I_z$ of type $\MNW$ there exists a unique
$V$-$\,$homomorphism $\varphi$ from $M\oV N$ to $W$ such that
$I_z=\varphi \circ F_z$, where $\varphi$ is extended canonically
to a linear map from $(M\oV N)((z))$ to $W((z))$, it follows that
$I_{_{(n)}}(u,v)=0$ for all $n\geq N$.\end{proof}

\begin{theorem} (Sufficient condition) If the pair $(M,N)$ satisfies the kernel
intertwining operator full equality condition, then the tensor
product $(M\oV N, F_z)$ exists.
\end{theorem}

We shall do the construction of the tensor product in the rest of
this section, and we shall see that it is also a sufficient
condition for the existence of the tensor product.

\begin{lemma} {\rm (Lemma 5.1.3, \cite{Li1})} Let $(M\oV N, F_z)$ be a
tensor product of the ordered pair $(M,N)$. Then $F_z$ is
surjective in the sense that all the coefficients of $F_z(u,v)$,
for $u\in M$ and $v\in N$, linearly span $M\oV N$.
\end{lemma}

With this lemma in mind, the essential idea is to consider
"strings over $\mathbb{Z}$" for each pair $(u,v)\in M\times N$ and
then take the quotient
 by all the necessary conditions in order to get an intertwining
 operator by taking the generating series of these strings (cf.
  \cite{Li1}).

Therefore, let $M$ and $N$ be two $V$-$\,$modules, and set
\begin{equation*}
F_0(M,N)=\mathbb{C}[t,t^{-1}]\tt M\tt N.
\end{equation*}
As usual in vertex algebra theory, it is more clear to work with
generating series in order to manipulate string of vectors. For
this reason we introduce the following very important notation for
any $u\in M$ and $v\in N$ (cf.
 formula (5.2.2) in \cite{Li1}): \vskip .2cm

\begin{equation}\label{z}
u \oz v:=\delta(z,t)\tt u\tt v=\sum_{n\in \mathbb{Z}} (t^n\tt u\tt
v) z^{-n-1}.
\end{equation}
We want to take the necessary quotients in order to obtain that
$\ozz\ $ is an intertwining operator. In particular, it should
satisfies the commutator formula (\ref{Icommut}), and this is the motivation for
the following action, but in fact, we will see that it is not a
Lie theoretical type action.

By using the generating series of elements in $\g(V)$ given by
\begin{equation*}
a\z =\delta(z,t)\tt a=\sum_{n\in \mathbb{Z}} (t^n\tt a) z^{-n-1},
\end{equation*}
we define an action of $\g(V)$ on $F_0(M,N)$ as follows (for $a\in
V,u\in M,v\in N$):

\begin{align}\label{acc}
    a\z(u\oy v)&=u\oy (a\z v)+(a\zmy u- a\myz u)\oy v\\
    & =u\oy (a\z v)+ \mathrm{Res}_x\
\de(z-x,y) (a\x u\oy v)\nonumber
\end{align}

\

\noi (cf. formulas of $\Delta_z$ in p. 183 \cite{MS} and (5.2.6) in
 \cite{Li1}), and define the linearly extended map given on generators by
$d(u\oz v)= du\oz v+ u\oz dv$.

\

\begin{proposition}\label{p2}
Under the above defined action, $F_0(M,N)$ is a $\g(V)$-module
satisfying
$$
\vac\z=\mathrm{id},\ \ (da)\z =\frac{d}{dz}a\z, \quad \mathrm{and
}\ \ [\, d\, ,\, a\z]=(da)\z
$$

\noindent Moreover, $F_0(M,N)$ is just the tensor product of Lie
algebra  modules $\widehat M$ and $ N$ (cf. Proposition \ref{lie}).
\end{proposition}

\begin{proof}
Expanding the coefficients in (\ref{acc}), we obtain:
\begin{equation}\label{action}
    (t^m\tt a)(t^n\tt u\tt v)=t^n\tt u \tt a_m v + \sum_{i=0}^\infty \binom{m}{i}
    t^{m+n-i}\otimes a_i u\tt v.
\end{equation}
Using Proposition \ref{lie}, we have the $\g (V)$-module structure
given by the tensor product of Lie algebra modules. The other
properties follow by straightforward computations.
\end{proof}

Let $J_0$ be the $\g(V)$-submodule of $F_0(M,N)$ generated by the
following subspace:

\begin{align}
    \kk\mathrm{-span}\ \bigg\{ t^n\tt u\tt v\in F_0(M, N)  \st  I_{_{(n)}}(u,v)&=0 \textrm{ for all }
 V\textrm{-modules $W$,}\nonumber
 %\\&
  \textrm{ and for all }I_z\in
  {\vcenter{\hbox{$\genfrac{(}{)}{0pt}{1}{W}{M,N}$}}}\bigg\}
  \nonumber
\end{align}

\

Since $M\times N=Ker\, \binom{\cdot}{M,N} $, then for every
$(u,v)\in M\times N$, there exists $l\in\mathbb{N}$ such that
$I_{_{(n)}}(u,v)=0$ for all $n\geq l$, $\textrm{ for all }
 V\textrm{-modules $W$,}
  \textrm{ and for all }I_z\in\MNW$.

Now, we take
\begin{equation*}   F_1(M,N)=F_0(M,N)/J_0.
\end{equation*}

\vskip .3cm

We will still use the notation $u\ox v$ and $t^n\tt u\tt v$ for
elements in the quotient space $F_1(M,N)[[x,x^{-1}]]$ and
$F_1(M,N)$ respectively. We have the following important result:

\vskip .3cm

\begin{proposition}\label{p1}
For any $a\in V$ and any $t^n\tt u\tt v\in F_1(M,N)$, we
have

\vskip .2cm

\noi (a) $a\z(t^n\tt u\tt v)$ involves only finitely many negative
powers of $z$.

\vskip .2cm

\noi (b) $u\ox v$ involves only finitely many negative powers of
$x$.

\vskip .2cm

\noi (c) Let $a\in V, u\in M$ and $v\in N$. Then for any integer
$k$ such that $k\geq N_{a,u}$ (see (\ref{nnnnn})), we have

\begin{equation*}
(z-x)^k \ a\z (u\ox v)= (z-x)^k \ (u\ox a\z v) \ \in
F_1(M,N)((z,x)).
\end{equation*}

\vskip .2cm

\noi (d) For any integer $k$ such that $k\geq N_{a,u}$, an
alternative definition for the action of $\g(V)$ (or $V$) in
$F_1(M,N)$ is given by
\begin{equation*}
\ a\z (u\ox v)= (z-x)^{-k} \big((z-x)^k \ (u\ox a\z v) \big).
\end{equation*}
similar to the ring theoretical case. \vskip .2cm

\noi (e) The map $d$ is well defined in $F_1(M,N)$, that is
$d(J_0)\subseteq J_0$.

\end{proposition}

\vskip .3cm

\begin{proof}
(a) We fix $a\in V$ and $t^n\tt u\tt v\in F_1(M,N)$. Recall
formula (\ref{action}):
$$
(t^m\tt a)(t^n\tt u\tt v)=t^n\tt u \tt a_m v + \sum_{i=0}^\infty
\binom{m}{i}
    t^{m+n-i}\otimes a_i u\tt v.
    $$
    Observe that in the first term, $a_m v=0$ for a sufficiently large $m$, and
    in the second term, the sum is finite and independent of $m$.
 Then for each element $a_i u\tt
    v$ there exists a power of $t$ such that $t^l\tt a_i u\tt
    v=0 $ if $l$ is large enough. Therefore, for a large enough
    $m$ the result is proved.

(b) It is immediate from the definition of $J_0$ and the kernel
condition that is assume for the pair $(M,N)$.

(c) Multiplying (\ref{acc}) by $(z-x)^{N_{a,u}}$, with $N_{a,u}$
as in  (\ref{nnnnn}), we obtain the equality. Since the LHS is an
element in $(F_1(M,N))((z))((x))$ and the RHS is in
$(F_1(M,N))((x))((z))$, we get the result.

(d) Note that the LHS in (c) involve only finitely many negative
power of $x$, then we multiply both sides by $(z-x)^{-k}$. On the
other hand, (c) implies that $(z-x)^k \ (u\ox a\z v) \ \in
F_1(M,N)((z,x))$, so that the expression on the RHS of (d) is well
defined. Now, it remains to prove that it is independent of $k$.
Assume that $k_1\geq k_2\geq N_{a,u}$. Then for $i=1,2$, we have
$(z-x)^{k_i} \ (u\ox a\z v) \ \in F_1(M,N)((z,x))$ and
\begin{align*}
(z-x)^{-k_1}((z-x)^{k_1} \ (u\ox a\z
v))&=(z-x)^{-k_1}((z-x)^{k_1-k_2}((z-x)^{k_2} \ (u\ox a\z v)))\\
&=(z-x)^{-k_1}(z-x)^{k_1-k_2}((z-x)^{k_2} \ (u\ox a\z v))\\
&=(z-x)^{-k_2}((z-x)^{k_2} \ (u\ox a\z v)),
\end{align*}
proving the assertion (cf. similar arguments to those in Remark
5.17 and the proof of Lemma 5.10 in \cite{Li3}).

(e) Suppose $t^n\tt u\tt v\in J_0$, then $d(t^n\tt u\tt v)=t^n\tt
du\tt v+t^n\tt u\tt dv$. Using (\ref{ID}), we have $I_{_{(n)}}((du, v)+
(u , dv))=d\ I_{_{(n)}}(u , v)=0$, finishing the proof.
\end{proof}

Now,  let $J_1$ be the $\g(V)$-submodule of $F_1(M,N)$ linearly
spanned by all the coefficients in the following expressions:

\begin{align*}%\label{a1}
    \dxyz \, a\x (u \oy v) - &\  \dmyxz\,  (u\oy a\x v)-
\dxzy\, (a \z u \oy v),\\
&(d\, u \oz v)- \ddz (u \oz v).%\label{a2}
\end{align*}
for $a\in V,u\in M,v\in N.$ By straightforward computations, it is
easy to see that $J_1$ is invariant by $d$. We define
\begin{equation*}%\label{ttt}
    M\oV N= F_1(M,N)/J_1.
\end{equation*}

\begin{proposition}
The space $M\oV N$ is a $V$-$\,$module, and $\ozz$ is an intertwining
operator of type $\binom{_{_{M \oV N}}}{^{_{M,N}}}$.
\end{proposition}

\begin{proof} By Proposition \ref{p1}, and the $J_1$-defining
relations, it is clear that $\ozz$ is an intertwining operator. In
order to see that $M\oV N$ is a $V$-$\,$module, by Proposition
\ref{p2}, it remains to check Jacobi identity following the proof
 of Theorem 5.2.9 in \cite{Li1}. For $a,b\in  V, u\in M, v\in N$, we have:

\begin{align}
&\de(z-w,x) \ a\z \, (b\w (u\oy v)) = \de(z-w,x) \ a\z\bigg(u\oy
b\w
    v+\mathrm{Res}_r \de(w-r,y)(b\r u \oy v)\bigg)\label{24}\\
    &=\de(z-w,x) \bigg[ u\oy a\z (b\w v)+\mathrm{Res}_r \de(z-r,y)
    (a\r u\oy b\w v)+\mathrm{Res}_r \de(w-r,y)a\z (b\r u\oy v)\bigg]\nonumber
\end{align}
and
\begin{align}
&\de(-w+z,x) \ b\w \, (a\z (u\oy v)) = \de(-w+z,x) \ b\w\bigg(u\oy
a\z
    v+\mathrm{Res}_r \de(z-r,y)(a\r u \oy v)\bigg)\label{27}\\
    &=\de(-w+z,x) \bigg[ u\oy b\w (a\z v)+\mathrm{Res}_r \de(w-r,y)
    (b\r u\oy a\z v) +\mathrm{Res}_r \de(z-r,y)b\w (a\r u\oy v)\bigg]\nonumber
\end{align}
and, on the other hand
\begin{equation}\label{30}
    \de(z-x,w)(a\x b)\w (u\oy v)=\de(z-x,w)\bigg[u\oy(a\x b)\w v +
    \mathrm{Res}_r \de(w-r,y)((a\x b)\r u\oy v)\bigg].
\end{equation}
It follows from the Jacobi identity of $N$ that ($1^{st}$ term in
(\ref{24}))$-$($1^{st}$ term in (\ref{27}))$=$($1^{st}$ term in
(\ref{30})). Using that
\begin{align*}
    \de(z-w,x)\de(w-r,y)&=\de(z-(y+r),x)\de(w-r,y)=\de(z-y,x+r)\de(w-r,y)\\
\de(-w+z,x)\de(w-r,y)&=\de(-(y+r)+z,x)\de(w-r,y)=\de(-y+z,x+r)\de(w-r,y)
\end{align*}
and the defining relations of $J_1$, we obtain that ($3^{rd}$ term
in (\ref{24}))$-$($2^{nd}$ term in (\ref{27})) is equal to:
\begin{align*}
\mathrm{Res}_r &\  \de(w-r,y)\big[\de(z-y,x+r) a\z(b\r u\oy
v)-\de(-y+z,x+r) (b\r  u\oy a\z v)\big]=\\
&=\mathrm{Res}_r \,\de(w-r,y)\de(z-(x+r),y)(a\xr (b\r u)\oy v)\\
&=\mathrm{Res}_r \,\mathrm{Res}_s \,\de(w-r,y)\de(z-(x+r),y)\de(x+r,s)(a\sss (b\r u)\oy v)\\
&=\mathrm{Res}_r \,\mathrm{Res}_s
\,\de(w-r,y)\de(z-s,y)\de(s-r,x)(a\sss (b\r u)\oy v).
\end{align*}
Similarly, we have that ($2^{nd}$ term in (\ref{24}))$-$($3^{rd}$
term in (\ref{27})) is equal to:
\begin{align*}
\mathrm{Res}_r &\  \de(z-r,y)\big[-\de(-y+w,-x+r) (a\r u\oy b\w
v)+\de(w-y,-x+r) b\w( a\r u\oy  v)\big]=\\
&=\mathrm{Res}_r \,\de(z-r,y)\de(w+(x-r),y)(b\mxr (a\r u)\oy v)\\
&=\mathrm{Res}_r \,\mathrm{Res}_s \,\de(z-r,y)\de(w-s,y)\de(-x+r,s)(b\sss (a\r u)\oy v)\\
&=-\mathrm{Res}_r \,\mathrm{Res}_s
\,\de(z-s,y)\de(w-r,y)\de(-r+s,x)(b\r (a\sss u)\oy v).
\end{align*}
Therefore, using Jacobi identity of $M$, we have that ($3^{rd}$
term in (\ref{24})$-2^{nd}$ term in (\ref{27})+$2^{nd}$ term in
(\ref{24})$-$$3^{rd}$ term in (\ref{27})) is equal to
\begin{align*}
\mathrm{Res}_r \,\mathrm{Res}_s &
\,\de(z-s,y)\de(w-r,y)\de(s-x,r)((a\x b)\r u\oy v)=\\
&=\mathrm{Res}_r
\,\de(z-(r+x),y)\de(w-r,y)((a\x b)\r u\oy v)\\
&=\mathrm{Res}_r
\,\de(y+(r+x),z)\de(w-r,y)((a\x b)\r u\oy v)\\
&=\mathrm{Res}_r
\,\de(w+x,z)\de(w-r,y)((a\x b)\r u\oy v)\\
&=\mathrm{Res}_r
\,\de(z-x,w)\de(w-r,y)((a\x b)\r u\oy v)\\
\end{align*}
which is the same as the second term in (\ref{30}), finishing the
proof.
\end{proof}

\

Now, it is clear from the construction that we have

\begin{theorem}  If the pair $(M,N)$ satisfies the kernel
intertwining operator full equality condition, then the pair
$(M\oV N, \ozz\, )$ is a tensor product of the pair $(M,N)$.
\end{theorem}

\

%%%%%%%%%%%%%%%%%%%%%%%%%%
\section{Second construction and a ring theoretical interpretation}\lbb{secondoo}
%%%%%%%%%%%%%%%%%%%%%%%%%%

\ 

One of the motivations for the definition of vertex bilinear map
is the following: let $R$ be a commutative (associative) ring,
then the product $\mu$ in $R$ is an $R$-bilinear map, that is
$$
\qquad \qquad \qquad \qquad r\mu(a,b)= \mu(ra,b)=\mu(a,rb) \qquad \hbox{ for }\ \ r,a,b\in R,
$$
which means
that
$$
r(ab)=(ra)b \quad \hbox{ and } \quad r(ab)=a(rb),
$$
observe that the first one corresponds to associativity (in our
case (\ref{AA})), and the second one is the commutativity of left
multiplication operators, that in our case corresponds to
locality, that is (\ref{BB}).

By standard arguments in vertex algebra theory, we have
that the following three definitions are equivalent to each other
and they are equivalent to the notion of intertwining operator
(see Theorem \ref{bbb} below).

\

\begin{definition} \label{bilin}
(1) Let $M,N$ and $W$ be three $V$-$\,$modules. A $\kk$-bilinear map
$F_z:M\times N \to W((z))$ is called a {\it vertex bilinear map}
or {\it $V$-$\,$bilinear of type $(M,N;W)$} if it satisfies:

\vskip .2cm

(a) $F_z(d u, v)=\frac{d}{dz} F_z(u,v)$ and $F_z( u,d
v)=\big(d-\frac{d}{dz}\big) F_z(u,v)$.

\vskip .2cm

(b) For any $a\in V$, $u\in M$ and $v\in N$, there exist $n,k\in
\mathbb{N}$ such that

\begin{align}\label{AA}
(z+x)^n \ a\zx F_x(u,v)&=(z+x)^n \ F_x(a\z u,v)
\end{align}
and
\begin{align}\label{BB}
(z-x)^k \ a\z F_x(u,v)=(z-x)^k \ F_x(u,a\z v).
\end{align}
\

\noi (2) Let $M,N$ and $W$ be three $V$-$\,$modules. A $\kk$-bilinear
map $F_z:M\times N \to W((z))$ is called a {\it vertex bilinear
map} or {\it $V$-$\,$bilinear of type $(M,N;W)$} if it satisfies:

\vskip .2cm

(a) $F_z(d u, v)=\frac{d}{dz} F_z(u,v)$ and $F_z( u,d
v)=\big(d-\frac{d}{dz}\big) F_z(u,v)$.

\vskip .2cm

(b) For any $a\in V$, $u\in M$ and $v\in N$

\begin{align}\label{Aa}
(z+x)^{N_{a,v}} \ a\zx F_x(u,v)&=(z+x)^{N_{a,v}} \ F_x(a\z u,v)
\end{align}
and
\begin{align}\label{Bb}
(z-x)^{N_{a,u}} \ a\z F_x(u,v)=(z-x)^{N_{a,u}} \ F_x(u,a\z v).
\end{align}
\

\noi (3) Let $M,N$ and $W$ be three $V$-$\,$modules. A $\kk$-bilinear
map $F_z:M\times N \to W((z))$ is called a {\it vertex bilinear
map} or {\it $V$-$\,$bilinear of type $(M,N;W)$} if it satisfies:

\vskip .2cm

(a) $F_z(d u, v)=\frac{d}{dz} F_z(u,v)$ and $F_z( u,d
v)=\big(d-\frac{d}{dz}\big) F_z(u,v)$.

\vskip .2cm

(b) For any $a\in V$, $u\in M$ and $v\in N$, there exist $n\in \NN$
depending on $a$ and $v$, and $k\in \NN$ depending on $a$ and $u$ such that

\begin{align}\label{AAa}
(z+x)^{n} \ a\zx F_x(u,v)&=(z+x)^{n} \ F_x(a\z u,v)
\end{align}
and
\begin{align}\label{BBb}
(z-x)^{k} \ a\z F_x(u,v)=(z-x)^{k} \ F_x(u,a\z v).
\end{align}

\end{definition}

\

We denote by $V$-$\,Bilinear(M,N;W)$ the space of all vertex
bilinear maps of this type. The other motivation for the
definition of vertex bilinear map is the following result:

\begin{theorem}\label{bbb} The three definitions of vertex bilinear
 map are equivalent to each other. And,
an intertwining operator of type $\MNW$ is the same as a vertex
bilinear map of type $(M,N;W)$.
\end{theorem}

\begin{proof}
It is clear that the Jacobi identity (\ref{Ijacobi}) imply (\ref{Aa})
and (\ref{Bb}) in the second definition of vertex bilinear map, by using the
 associator formula (\ref{Iassoc}) and the commutator formula (\ref{Icommut}).
 It is also clear that the second definition imply the third, and the third imply
  the first one. So, it  remains to show that (\ref{AA}) and (\ref{BB}) imply
   the Jacobi identity (\ref{Ijacobi}). But, this is immediate by
   using the following lemma (cf. Proposition 3.4.3 in \cite{LL}):
\end{proof}

\begin{lemma}\label{l1}
{\rm (Lemma 2.1, \cite{Li3})} Let $U$ be a vector space and let
\begin{align*}
A(x,y)\in & U((x))((y))\\
B(x,y)\in & U((y))((x))\\
C(z,y)\in & U((y))((z)).
\end{align*}
Then
$$
\delta(x-y,z) \ A(x,y)-\delta(-y+x,z)\ B(x,y)=\delta(x-z,y)\
C(z,y)
$$
\vskip .2cm

\noindent if and only if  there exist $k,l\in \mathbb{N}$ such
that
\begin{align*}
(x-y)^k A(x,y) &=(x-y)^k B(x,y)\\
(z+y)^l A(z+y,y) &= (z+y)^l C(z,y).
\end{align*}
\end{lemma}

\

In this way we try to eliminate the Jacobi identity and the notion
of intertwining operator as it is usually defined.

Now, we can redefine the notion of tensor product with a ring
theoretical interpretation:

\definition\label{tensor2}
Let $M$ and $N$ be two $V$-$\,$modules. A pair $(M\oV N, F_z)$, which
consists of a $V$-$\,$module $M\oV N$ and a vertex bilinear map $F_z$
of type $(M,N;M\oV N)$,  is called a {\it tensor product} for the
ordered pair $(M,N)$ if the following universal property holds:
For any $V$-$\,$module $W$ and any vertex bilinear map $I_z$ of type
$(M,N;W)$, there exists a unique $V$-$\,$homomorphism $\varphi$ from
$M\oV N$ to $W$ such that $I_z=\varphi \circ F_z$, where $\varphi$
is extended canonically to a linear map from $(M\oV N)((z))$ to
$W((z))$.

\

Using this universal property of the tensor product with respect
to the vertex bilinear maps, we have:

\begin{proposition}\label{BBBB}
The space $V$-$\,$bilinear$(M,N;W)$ is linearly isomorphic to
$Hom_V(M\oV N, W)$.
\end{proposition}

\

Now, we start our second (and equivalent) construction of the
tensor product. Take $F_1(M,N)$ as before. Using the equivalence
between the Jacobi identity with the identities (\ref{Aa}) and
(\ref{Bb}), we can redefine $J_1$ as the $\g(V)$-submodule of
$F_1(M,N)$ linearly spanned by all the coefficients in the
following expressions:
\begin{align*}
    (z+x)^{N_{a,v}}\ a\zx (u\ox v)&-(z+x)^{N_{a,v}} \ (a\z u\ox v)\\
    (z-x)^{N_{a,u}}\ a\z (u\ox v)&-(z-x)^{N_{a,u}} \ (u\ox a\z v)\\
(d\, u \ox v) &- \ddx (u \ox v).
\end{align*}

\noi for $a\in V$, $u\in M$, $v\in N$, and in this way we complete the
picture of this ring theoretical   construction.

This second construction is clearly the extension of the well
known tensor product of modules over an associative commutative
ring.

\vskip .7cm

\begin{example}
Let $V$ be an associative commutative algebra with unit. It
is clear that the second construction of
tensor product coincides with the usual tensor product of modules
over a commutative ring.
\end{example}

\begin{example} \label{ex}
We may restrict our definitions and constructions to a graded
vertex algebra $V$, and the category $C^+(V)$ of $\ZZ_+$-$\,$graded
$V$-$\,$modules, together with the intertwining operators that
preserve the gradations (cf. p.132 \cite{FZ} and p.56 \cite{FHL}),
in the sense that if $I_z\in \MNW$ with
$$
I_z(u,v)=\sum_{n\in \ZZ} I_{(n)}(u,v) z^{-n-1},
$$
then for homogeneous elements $u\in M$ and $v\in N$,  it satisfies
$$
I_{(n)}(u,v)\in W(\hbox{deg }u+\hbox{deg }v-n-1).
$$
Since $W\in C^+(V)$, we have that
$$
I_{(k)}(u,v)=0 \quad \hbox{for all }\ \ \  k\geq \hbox{deg }u+\hbox{deg
}v,
$$
proving that pair $(M,N)$ satisfies the kernel intertwining
operator full equality condition. Thus, there exists the tensor
product of them, which is also $\ZZ_+$-graded by defining
deg$(t^n \otimes u\otimes v )=$ deg$(u) + $ deg$(v)-n-1$.

In particular it holds for the category $C^+(V)$ of all known
rational vertex operator algebras $V$ such as $V_L$, associated
with a positive definite even lattice $L$; $L(l,0)$, associated
with an affine Lie algebra with a positive integral level $l$;
$L(c_{p,q},0)$, associated with the Virasoro algebra with central
charge $c_{p,q}$ in the minimal series; and $V^\natural$, the
moonshine module.
\end{example}

\

Let  \textsf{Mod}$_V$  be the category of $V$-$\,$modules with the
maps given by $Hom_V(M,N)$ for $V$-$\,$modules $M$ and $N$. From now
on we shall assume the existence of all the tensor product of
modules that appear. One can restrict all the following results
to certain subcategory where the tensor products always exist.

\begin{proposition}
Let $f:M\to M\s '$ and $g:N\to N\s '$ be homomorphisms of
$V$-$\,$modules. Then there is a unique $V$-$\,$homomorphism denoted by
$f\otimes g:M\oV N\to M\s '\oV N\s '$, with
\begin{equation*}
(f\otimes g)(u\oz v)=f(u)\oz g(v)
\end{equation*}
for all $u\in M$ and $v\in N$.
\end{proposition}

\begin{proof}
The function $F_z:M\times N\to (M\, '\oV N\, ')((z))$ given by
$F_z(u\, ,v)=f(u)\oz \,g(v)$, is easily seen to be a vertex bilinear
map. Then, by universal property, there is a unique
$V$-$\,$homomorphism from $M\oV N$ to $M\s '\oV N\s '$ sending $u\oz
v$ to $f(u)\oz g(v)$.
\end{proof}

\begin{corollary}\label{coco}
Given $V$-$\,$homomorphisms $M\xrightarrow{f} M\s ' \xrightarrow{f'}
M\s ''$ and $N\xrightarrow{g} N\s ' \xrightarrow{g'} N\s ''$, we
have
\begin{equation*}
(f'\otimes g')(f\otimes g)=f' f \otimes g' g.
\end{equation*}
\end{corollary}

\begin{proof}
Both maps take $u\oz v\mapsto f'(f(u))\oz g'(g(v))$, and so the
uniqueness of such a homomorphism gives the desired equation.
\end{proof}

\begin{theorem}
(a) Given $M$ a $V$-$\,$module, there is a covariant functor
$F_M:\, $\textsf{Mod}$_V \to $ \textsf{Mod}$_V$ defined by
\begin{equation*}
F_M(N)=M\oV N \quad \hbox{ and } \quad F(g)=1_M\otimes g
\end{equation*}
for a $V$-$\,$homomorphism $g:N\to N\s '$.

\

\noi (b) Given $N$ a $V$-$\,$module, there is a covariant functor
$G_N:\,$\textsf{Mod}$_V \to $ \textsf{Mod}$_V$ defined by
\begin{equation*}
G_N(M)=M\oV N \quad \hbox{ and } \quad G(f)=f\otimes 1_N
\end{equation*}
for a $V$-$\,$homomorphism $f:M\to M\s '$.
\end{theorem}

\vskip .1cm

\begin{proof}
First, note that $F_M$ preserves identities: $F_M(1_N)=1_M\otimes
1_N$ is the identity $1_{M\oV N}$, because it fixes every
generator. Second, $F_M$ preserves composition:
\begin{equation*}
F_M(g' g)=1_M\otimes g' g=(1_M\otimes g' )(1_M\otimes
g)=F_M(g')F_M(g)
\end{equation*}
by Corollary \ref{coco}. Therefore $F_M$ is a covariant functor.
In a similar way one can prove (b).
\end{proof}

We denote the functor $F_M$ by $M\oV \square$, and the functor
$G_N$ by $\square\oV N$. We shall work with the tensor product
"functor" when in fact we need to assume the existence.

\

%%%%%%%%%%%%%%%%%%%%%%%%%%%%%%%%%%%%%%%%%%%%%

\section{Vertex homomorphism and  adjoint isomorphisms}\lbb{second}

%%%%%%%%%%%%%%%%%%%%%%%%%%%%%%%%%%%%%%%%%%%%%

\

In this section we introduce the notion of (right) vertex
homomorphism, producing  the "$Hom$" functor for modules over a
vertex algebra, obtaining (from our point of view) an analog of
the $Hom$ functor for modules over an associative commutative
algebra. We prefer to think about it in this way, instead of the
$Hom$ functor in Lie theory. In the first part of this section, we
follow \cite{Li1} and \cite{Li2}.

The most important functors in homological algebra are $Hom$,
tensor product and functors derived from them. In the case of
vertex algebras, we should consider $Vhom$ (and later $Vhom^r$).
We shall show that there is an intimate relationship between
$Vhom$ and tensor, they form an adjoint pair of functors.

\begin{definition} \label{Vhom}
Let $M$ and $N$ be two $V$-$\,$modules. A {\it
vertex homomorphism} from $M$ to $N$ is a liner map
$f_z:M\longrightarrow N((z))$ such that

\vskip .3cm

(a) $f_z(d \, u)=\big(d-\frac{d}{dz}\big)\, f_z(u)$ \ for all $u\in M$.

\vskip .3cm

(b) For any $a\in V$, there exists $k\in\mathbb{N}$ such that
\begin{equation} \label{left-b}
(z-x)^k\ a\Nz (f_x(u))=(z-x)^k \ f_x(a\Mz u)\qquad \textrm{for all }
u\in M.
\end{equation}

\end{definition}

\vskip .3cm

We denote by $Vhom(M,N)$ the space of all vertex homomorphisms
from $M$ to $N$. Define $d$ on $Vhom(M,N)$ given by $(d\,
f)_z=\frac{d}{dz}(f_z)$. Observe that if $V$ is an associative
commutative algebra with unit and $d=0$, and $M$ and $N$ are
modules over it with the corresponding derivations equal to zero,
then (using (a) in the definition) any vertex homomorphism must be
independent of $z$ and $Vhom(M,N)=Hom_V(M,N)$, the usual
homomorphisms of modules over an associative commutative algebra.

One of the motivation for the definition of vertex homomorphism is
the following: let $g_z:M\times N\to W((z))$ be a vertex bilinear
map. Then for any fixed  $u\in M$, the map $f_z(v)=g_z(u,v)$ is a
vertex homomorphism from $N$ to $W$.

Now, we define an action of $V$ on $Vhom(M,N)$ as follows:

\begin{align}\label{vhom-action}
    (a\z f)_y(u)=a\Nzy(f_y(u))-f_y(a\Mzy u - a\Myz u)
\end{align}
which can also be rewritten as (cf. with associator formula
(\ref{assoc}))
\begin{align*}
(a\z f)_y(u)=\mathrm{Res}\, _x \,\Big( \de(x-y,z) a\Nx (f_y(u))-
\de(-y+x,z)f_y(a\Mx u)\Big),
\end{align*}
for $a\in V$ and $u\in M$.

\

\begin{proposition} \label{cheta} Let $a\in V$, $f\in Vhom(M,N)$ and $u\in M$. Then we
have the following properties:

\vskip .3cm

(a) \ $a\z f\in Vhom(M,N)((z))$.

\vskip .3cm

(b) \ $\mathbf{1}\z f=f$.

\vskip .3cm

(c) \  $(da)\z f=\frac{d}{dz}(a\z f)$.

\vskip .3cm

(d) \  $d(a\z f)=(da)\z f + a\z (d\, f)$.

\vskip .3cm

(e) \  For any integer $k$ such that $k\geq N_{a,u}$, we have

\begin{equation*}
(z+y)^k \ (a\z f)_y (u)= (z+y)^k \ a\Nzy(f_y(u)) \ \in N((z,y)).
\end{equation*}

\vskip .2cm

(f) \  For any integer $k$ such that $k\geq N_{a,u}$, an alternative
definition for the action of $V$ in $Vhom(M,N)$ is given by
\begin{equation*}
\ (a\z f)_y(u)= (y+z)^{-k} \big((z+y)^k \ a\Nzy(f_y(u)) \big).
\end{equation*}
similar to the ring theoretical case.

\end{proposition}

\begin{proof}
(a) If we write $(a\z f)_y=\sum_{n\in\mathbb{Z}}\, (a\,_n f)_y\,
z^{-n-1}$, then
\begin{align}\label{44}
(a\,_n f)_y=\mathrm{Res}_{\, x} \,\Big( (x-y)^n a\Nx (f_y)- (-y+x)^n
f_y(a\Mx)\Big).
\end{align}
Therefore, there exists $N\in \mathbb{N}$ such that $(a\,_n
f)_y=0$ for $n\geq N$.

Now, we should prove that $(a\,_n f)_y(du)=(d-\frac{d}{dy})((a\,_n f)_y(u))$
for all $n\in\ZZ$,
or equivalently $(a\z f)_y(du)=(d-\frac{d}{dy})((a\z f)_y(u))$, that follows by
\begin{align*}
    \frac{d}{dy}\Big((a\z f)_y(u)\Big) & = \frac{d}{dy}
    \mathrm{Res}\, _x \,\Big[ \de(x-y,z) a\Nx (f_y(u))-
\de(-y+x,z)f_y(a\Mx u)\Big]\\
& = \mathrm{Res}\, _x \,\Big[ \Big(\frac{d}{dy}\de(x-y,z)\Big) a\Nx (f_y(u))-
\Big(\frac{d}{dy}\de(-y+x,z) \Big )f_y(a\Mx u)\Big]\\
& \ \ \ +\mathrm{Res}\, _x \,\Big[ \de(x-y,z) a\Nx \Big(\frac{d}{dy}\ f_y(u)\Big)-
\de(-y+x,z)\frac{d}{dy}\Big(f_y(a\Mx u)\Big)\Big]\\
& = - \mathrm{Res}\, _x \,\Big[ \Big(\frac{d}{dx}\de(x-y,z)\Big) a\Nx (f_y(u))-
\Big(\frac{d}{dx}\de(-y+x,z) \Big )f_y(a\Mx u)\Big]\\
& \ \ \ +\mathrm{Res}\, _x \,\Big[ \de(x-y,z) a\Nx \Big(\frac{d}{dy}\ f_y(u)\Big)-
\de(-y+x,z)\frac{d}{dy}\Big(f_y(a\Mx u)\Big)\Big]\\
& =  \mathrm{Res}\, _x \,\Big[ \de(x-y,z) \frac{d}{dx}\Big(a\Nx (f_y(u))\Big)-
\de(-y+x,z) f_y\Big(\frac{d}{dx}(a\Mx u)\Big )\Big]\\
& \ \ \ +\mathrm{Res}\, _x \,\Big[ \de(x-y,z) a\Nx \Big(\frac{d}{dy}\ f_y(u)\Big)-
\de(-y+x,z)\frac{d}{dy}\Big(f_y(a\Mx u)\Big)\Big]
\end{align*}

\begin{align*}
& = -\mathrm{Res}\, _x \,\Big[ \de(x-y,z) a\Nx (f_y(du))-
\de(-y+x,z)f_y(a\Mx (du))\Big]\\
& \ \ \ +\mathrm{Res}\, _x \,\Big[ \de(x-y,z) d\big(a\Nx (f_y(u))\big)-
\de(-y+x,z) d\big(f_y(a\Mx u)\big)\Big]\\
& = d\big((a\z f)_y(u)\big) - (a\z f)_y(du).
\end{align*}

Observe that (\ref{44}) is the
$(n)$-product of the fields $a\y$ and $f_y$. Using that the pairs
$(f_y,b\y)$ and $(a\y,b\y)$ are local, together with Dong's lemma,
we obtain that $b\y$ and $(a\,_n f)_y$ are local. Thus $(a\,_n
f)_y \in Vhom(M,N)$, finishing (a). Part (b) is trivial.

(c) For $a\in V,u\in M, f\in Vhom(M,N)$, we have
\begin{align*}
    ((da)\z f)_y(u)&=\mathrm{Res}_{\,x}\,
    \Big(\de(x-y,z)(da)\Nx(f_y(u)) - \de(-y+x,z)f_y((da)\Mx u)
    \Big)\\
&=\mathrm{Res}_{\,x}\,
    \Big(\de(x-y,z)\frac{d}{dx}(a\Nx(f_y(u))) - \de(-y+x,z)\frac{d}{dx}(f_y(a\Mx
    u))
    \Big)\\
&=\frac{d}{dz}\mathrm{Res}_{\,x}\,
    \Big(\de(x-y,z)\  a\Nx(f_y(u)) - \de(-y+x,z)f_y(a\Mx u)
    \Big)=\frac{d}{dz}(a\z f)_y(u).\\
\end{align*}

(d) For $a\in V,u\in M, f\in Vhom(M,N)$, we have
\begin{align*}
(d(a\z f))_y(u)-(a\z(df))_y&(u)=\frac{d}{dy}(a\z
f)_y(u)-(a\z(df))_y(u)\\
&=\frac{d}{dy}\mathrm{Res}_{\,x}\,
    \Big(\de(x-y,z) \ a\Nx(f_y(u)) - \de(-y+x,z)f_y(a\Mx u)
    \Big)\\
    &\ \ -\mathrm{Res}_{\,x}\,
    \Big(\de(x-y,z)\ a\Nx \Big(\frac{d}{dy}f_y(u)\Big) - \de(-y+x,z)\frac{d}{dy}f_y(a\Mx u)
    \Big)\\
&=\mathrm{Res}_{\,x}\,
    \Big(\Big(\frac{d}{dy}\de(x-y,z)\Big)(a\Nx(f_y(u))) - \Big(\frac{d}{dy}\de(-y+x,z)\Big)f_y(a\Mx u)
    \Big)\\
    &=\frac{d}{dz}((a\z f)_y(u))=((da)\z f)_y(u).
\end{align*}

The proof of (e) and (f) is similar to the proof of (c) and (d) in
Proposition \ref{p1}.
\end{proof}

\begin{theorem}
$Vhom(M,N)$ is a $V$-$\,$module.
\end{theorem}

\begin{proof}
We follow the proof of Theorem 6.1.7 in \cite{Li2}. We only
need to prove the Jacobi identity. For any $a,b\in V$ and $f\in
Vhom(M,N)$, there exists $n\in \mathbb{N}$ such that

\begin{align*}
    (y-z)^n a\Ny (f_z)&=(y-z)^n f_z(a\My )\\
(t-z)^n \, b\Nt (f_z) &=(t-z)^n f_z(b\Mt )\\
(y-t)^n a\My(b\Mt) &=(y-t)^n b\Mt(a\My)\\
(y-t)^n a\Ny(b\Nt) &=(y-t)^n b\Nt(a\Ny).
\end{align*}
By definition, we have
\begin{align*}
    (a\w(b\x f))_z&=\mathrm{Res}_y\, \de(y-z,w) a\Ny((b\x
    f)_z)-\mathrm{Res}_y \, \de(-z+y,w) (b\y f)_z(a\My)\\
    &=\mathrm{Res}_y\mathrm{Res}_t\, (A-B-C+D),
\end{align*}
where
\begin{align*}
    A&=\de(y-z,w)\de(t-z,x) a\Ny(b\Nt(f_z))\\
B&=\de(y-z,w)\de(-z+t,x) a\Ny(f_z(b\Mt))\\
C&=\de(-z+y,w)\de(t-z,x) b\Nt(f_z(a\My))\\
D&=\de(-z+y,w)\de(-z+t,x) f_z(b\Mt(a\My)).\\
\end{align*}

\noi Similarly, we have
\begin{equation*}
(b\x(a\w f))_z=\mathrm{Res}_y\mathrm{Res}_t\, (A'-B'-C'+D'),
\end{equation*}
where
\begin{align*}
    A'&=\de(y-z,w)\de(t-z,x) b\Nt(a\Ny(f_z))\\
B'&=\de(y-z,w)\de(-z+t,x) b\Nt(f_z(a\My))\\
C'&=\de(-z+y,w)\de(t-z,x) a\Ny(f_z(b\Mt))\\
D'&=\de(-z+y,w)\de(-z+t,x) f_z(a\My(b\Mt)).\\
\end{align*}

By the properties of the $\de$-function, we obtain
\begin{equation*}
\de(w-x,r)r^nw^nx^n \, Q=\de(w-x,r)(y-z)^n (t-z)^n (y-t)^n \, Q
\end{equation*}
for any $Q\in\{A,B,C,D\}$. Therefore
\begin{align*}
    \de(w-x,r)r^nw^nx^n \, (A-B-C+D) &=\de(w-x,r)(y-z)^n (t-z)^n (y-t)^n \,
    (A-B-C+D)\\
    & =\mathrm{Res}_y\mathrm{Res}_t\,
    \de(w-x,r)\de(y-w,z)\de(t-x,z) \, X,
\end{align*}
where
\begin{equation*}
X=(y-z)^n (t-z)^n (y-t)^n \, a\Ny(b\Nt(f_z)).
\end{equation*}
Similarly, we have
\begin{align*}
    \de(-x+w,r)r^nw^nx^n \, (A'-B'-C'+D')  =\mathrm{Res}_y\mathrm{Res}_t\,
    \de(-x+w,r)\de(y-w,z)\de(t-x,z) \, X.
\end{align*}
Thus
\begin{align}\label{22}
    r^nw^nx^n \bigg[\bigg(\de(w-x,r) a\w(b\x)- &\de(-x+w,r) b\x(a\w)\bigg)\, f\, \bigg]_z=\\
    & =\mathrm{Res}_y\mathrm{Res}_t\,
    \de(w-r,x)\de(y-w,z)\de(t-x,z) \, X.\nonumber
\end{align}

On the other hand, we have by definition
\begin{align*}
    r^nw^nx^n & \de(w-r,x) \big((a\r b)\x f\big)_z=
\\
&=\mathrm{Res}_t r^nw^nx^n \de(w-r,x)\bigg(\de(t-z,x)(a\r
b)\Nt(f_z)-\de(-z+t,x)f_z\big((a\r b)\Mt\big)\bigg).
\end{align*}
But, using that for $W=M$ or $W=N$ we have:
\begin{align*}
    r^n(a\r b)\Wt &= \mathrm{Res}_y \, r^n\bigg(
    \de(y-t,r)a\Wy(b\Wt)-\de(-t+y,r) b\Wt(a\Wy)\bigg)\\
&= \mathrm{Res}_y \, \bigg(
    \de(y-t,r)(y-t)^n a\Wy(b\Wt)-\de(-t+y,r)(y-t)^n b\Wt(a\Wy)\bigg)\\
&= \mathrm{Res}_y \,
    \de(y-r,t)(y-t)^n a\Wy(b\Wt),
\end{align*}
then, we obtain
\begin{align}
   & r^nw^nx^n  \de(w-r,x) \big((a\r b)\x f\big)_z=\nonumber
\\
&=\mathrm{Res}_t \,\mathrm{Res}_y \,
\de(w-r,x)\de(t-z,x)\de(y-r,t) X-\mathrm{Res}_t \,\mathrm{Res}_y
\, \de(w-r,x)\de(-z+t,x)\de(y-r,t) X\nonumber\\
&=\mathrm{Res}_t \,\mathrm{Res}_y \,
\de(w-r,x)\de(t-x,z)\de(y-r,t) X \label{33}
\end{align}
By using the properties of $\de$, we obtain that the product of
the $\de$'s in (\ref{33}) is
\begin{align*}
    \de(w-r,x)&\de(t-x,z)\de(y-r,t)=\de(w-r,x)\de(z+x,t)\de(t+r,y)=\\
&=\de(w-r,x)\de(z+x,t)\de((z+x)+r,y)=\de(w-r,x)\de(z+x,t)\de(z+w,y)=\\
&=\de(w-r,x)\de(t-x,z)\de(y-w,z).
\end{align*}
Therefore, using (\ref{22}) and (\ref{33}), we have
\begin{align*}
    r^nw^nx^n  \de(w-r,x) \big((a\r b)\x f\big)_z=r^nw^nx^n
    \bigg[\bigg(\de(w-x,r) a\w(b\x)- &\de(-x+w,r) b\x(a\w)\bigg)\, f\, \bigg]_z
\end{align*}
Multiplying both sides by $r^{-n}w^{-n}x^{-n}$, we obtain the
Jacobi identity.
\end{proof}

\

Let $M,M\s ',N$ be $V$-$\,$modules. Given $f\in Hom_V(M\s ',M)$ we
define the induced map
\begin{equation*}
f_*:Hom_V(N,M\s ')\to Hom_V(N,M)
\end{equation*}
given by $[f_*(\psi)](v)=f(\psi(v))$ for all $v\in N$. Similarly,
we define
\begin{equation*}
f^*:Hom_V(M,N)\to Hom_V(M\s ',N)
\end{equation*}
given by $[f^*(\phi)](u')=\phi(f(u'))$ for all $u'\in M\s '$. By
standard arguments one can see that  $Hom_V(N,\square)$ (resp.
$Hom_V(\square,N)$) is a covariant (resp. contravariant) functor
from \textsf{Mod}$_V$ to \textsf{Vect}$_{\,\kk}$ (the category of
$\kk$-vector spaces).

\

Now, we are interested in $Vhom$. Let $M,M\s\; '$ and $N$ be
$V$-$\,$modules. Given $f\in Hom_V(M{\s } ',M)$, we can define an induced
map
\begin{equation*}
f_*: Vhom(N,M\s ')\to Vhom(N,M)
\end{equation*}
given by $[f_*(\psi)]_z(v)=f(\psi_z(v))$ for all $v\in N$.
Similarly, we define
\begin{equation*}
f^*: Vhom(M,N)\to Vhom(M\s ',N)
\end{equation*}
given by $[f^*(\phi)]_z(u')=\phi_z(f(u'))$ for all $u'\in M\s '$.
A simple computation shows that $f_*(\psi)\in Vhom(N,M)$ and
$f^*(\phi)\in Vhom(M\s ', N)$.

Observe that we use the same notation  $f_*$ and $f^*$ for
$Vhom$ and $Hom_V$.

\begin{proposition}\label{54}
(a) The maps $f_*$ and $f^*$ are $V$-$\,$homomorphisms.

\vskip .2cm

\noi (b) $Vhom(N,\square)$ is a covariant functor from
\textsf{Mod}$_V$ to \textsf{Mod}$_V$.

\vskip .2cm

\noi (c) $Vhom(\square,N)$ is a contravariant functor from
\textsf{Mod}$_V$ to \textsf{Mod}$_V$.
\end{proposition}

\begin{proof}
(a) It is clear that $f_*$ is additive, and
\begin{equation*}
[f_*(d\psi)]_z(v)=f((d\psi)_z(v))=\frac{d}{dz}f(\psi_z(v))=[d(f_*(\psi))]_z(v).
\end{equation*}
So, we should check that $f_*(a\z \psi)=a\z [f_*(\psi)]$ for any
$a\in V$. But, for all $v\in N$, we have
\begin{equation*}
[f_*(a\z \psi)]_x(v)=f\big((a\z
\psi)_x(v)\big)=f\big(a\zx(\psi_x(v))-\psi_x(a\zx v - a\xz
v)\big).
\end{equation*}
On the other hand, we have
\begin{align*}
     \big(a\z[f_*(\psi)]\big)_x(v) & =
     a\zx\big([f_*(\psi)]_x(v)\big )- [f_*(\psi)]_x(a\zx v -a\xz
     v)\\
     & = a\zx \big( f(\psi_x(v))\big)-f\big( \psi_x(a\zx v- a\xz
     v)\big) \\
     & = f\big(a\zx(\psi_x(v))\big )- f\big (\psi_x(a\zx v - a\xz
v)\big),
\end{align*}
proving that $f_*$ is a $V$-$\,$homomorphism. Similarly, by a simple
computation, one can check
that $f^*$ is also a $V$-$\,$homomorphism.

By standard computations one can see that these are functors,
proving (b) and (c). For example, if $f\in Hom_V(M\s ', M)$ and
$g\in Hom_V(M,M\s '')$, then for $\psi\in Vhom_V(N,M\s ')$ we have
\begin{align*}
  \big[g_*\big(f_*(\psi)\big)\big]_z(v)=g\big([f_*(\psi)]_z(v)\big
  )=g\big(f(\psi_z(v))\big)=(gf)(\psi_z(v))=\big[(gf)_*(\psi)\big]_z(v),
\end{align*}
for all $v\in N$.
\end{proof}

Now, we shall present some properties of $Vhom$ (cf. Proposition
6.2.5 in \cite{Li2}).

\begin{proposition}
Let $M$ be a $V$-$\,$module. Then $Vhom(V,M)$ is isomorphic to $M$ as
a $V$-$\,$module. The isomorphism is given by $\Psi_M:Vhom(V,M)\to M$
with $\Psi_M(f)=\hbox{Res}_{\,z}\ z^{-1} \, f_z(\vac)$. Moreover, if
$g\in Hom_V(M,N)$, then the following diagram commutes:

$$\begin{CD}
 Vhom(V,M) @>\Psi_M>> M \\
 @Vg_*VV  @VVgV\\
 Vhom(V,N) @>>\Psi_N> N
\end{CD}$$

\noi proving that $\Psi$ is a natural isomorphism from
$Vhom(V,\square)$ to the identity functor on \textsf{Mod}$_V$.
\end{proposition}

\begin{proof}
We define a map
\begin{equation*}
\Phi:M\longrightarrow Vhom(V,M)
\end{equation*}
given by $[\Phi(u)]_z(a)=e^{zd}(a\mz u)$, for $u\in M$. We should
check that $[\Phi(u)]_z\in Vhom(V,M)$. By definition, we have
\begin{align*}
    [\Phi(u)]_z(da)&=-e^{zd}\Big(\frac{d}{dz}(a\mz u)\Big)=
    \Big(\frac{d}{dz}e^{zd}\Big)
    (a\mz u)-\frac{d}{dz}(e^{zd}(a\mz u))\\
    & =\Big(d-\frac{d}{dz}\Big)(e^{zd}(a\mz
    u))
    =\Big(d-\frac{d}{dz}\Big)[\Phi(u)]_z(a).
\end{align*}
Using associativity, for any $b\in V$, there is $k\in \mathbb{N}$
depending only on $b$ and $u$ (see Remark \ref{module}) such that
\begin{equation*}
(z+x)^k b\zx (a\x u)=(z+x)^k (b\z a)\x u
\end{equation*}
for all $a\in V$. Thus,
\begin{align*}
    (z+x)^k b\zx \big( e^{xd}[\Phi(u)]_{-x}(a)\big)=(z+x)^k e^{xd}
    \big ([\Phi(u)]_{-x}(b\z
    a)\big).
\end{align*}
Using conjugation by $e^{xd}$, we get $(z+x)^k b\z
([\Phi(u)]_{-x}(a))=(z+x)^k [\Phi(u)]_{-x}(b\z a)$ for all $a\in
V$, proving that $[\Phi(u)]_z\in Vhom(V,M)$. It is easy to see
using the commutator formula for modules (\ref{commutator-m}) that
the linear map $\Phi$ from $M$ to $Vhom(V,M)$ is a
$V$-$\,$homomorphism.

Conversely,  given any $f_z\in Vhom(V,M)$, if we write
 $f_z=\sum_{n\in \mathbb{Z}} f_n
z^{-n-1}$, then by definition $f_n(da)=d(f_n(a))+n f_{n-1}(a)$. In
particular,
\begin{equation}\label{77}
d(f_n(\mathbf{1}))=-n f_{n-1}(\mathbf{1}).
\end{equation}
But, $f_z(\mathbf{1})$ is a Laurent serie, hence
$f_m(\mathbf{1})=0$ for some nonnegative $m$. Therefore, it
follows that $f_n(\mathbf{1})=0$ for all nonnegative $n$. Thus,
$f_z(\mathbf{1})$ involves only nonnegative powers of $z$, and
using (\ref{77}), one can easily get that
$f_z(\mathbf{1})=e^{zd}u$, where
$u=\Psi_M(f)=f_{-1}(\mathbf{1})\in M$. Then
\begin{align*}
    \Psi_M(a\x f)& =\hbox{Res}_{\,z} \, z^{-1} (a\x f)_z(\vac)=
    \hbox{Res}_{\,z} \, z^{-1} a\xz (f_z(\vac))=
    \hbox{Res}_{\,z} \, z^{-1} a\xz (e^{zd}\, u)\\
    & =a\x u
    = a\x(\Psi_M(f)),
\end{align*}
proving that $\Psi_M$ is a $V$-$\,$homomorphism.
Now, for any $a\in V$ there
exists $k\in \mathbb{N}$ such that $(z-x)^k f_z(a\x)=(z-x)^k
a\x(f_z)$. Then
\begin{align*}
    z^k f_z(a)& =  (z-x)^k f_z(a\x \mathbf{1})\, _{|_{x=0}}=
     (z-x)^k a\x(f_z( \mathbf{1}))\, _{|_{x=0}}= (z-x)^k a\x (e^{zd}u)
     \, _{|_{x=0}}\\
& = (z-x)^k e^{zd}(a\xmz u)\, _{|_{x=0}}=z^k e^{zd}(a\mz u).
\end{align*}
Thus, $f=\Phi(\Psi_M(f))$. Therefore $Vhom(V,M)$ is isomorphic to
$M$ as a $V$-$\,$module. Finally, we have
\begin{equation*}
\Psi_N(g_*(\psi))=\hbox{Res}_z \, z^{-1} \, [g_*(\psi)]_z(\vac)=
\hbox{Res}_z \, z^{-1} \, g(\psi_z(\vac))=g(\Psi_M(\psi)),
\end{equation*}
finishing the proof.
\end{proof}

\

The intertwining operators of type $\MNW$ are usually related to
$Hom_V(M,Vhom(N,W))$ and $Hom_V(N,Vhom^r(M,W))$ as in \cite{Li1}
 and \cite{DLM}. But we shall eliminate the notion of
intertwining operator in the statement of Theorem \ref{TT} (and
Theorem \ref {TT22} for $Vhom^r$), following the ring theoretical
point of view. The following result could be thought as a
motivation for the definition of $Vhom$ and the action of $V$ on
it.

\

\begin{theorem}\label{TT} {\rm (Adjoint isomorphism, first version)}
Let $M,N$ and $W$ be $V$-$\,$modules. Then there is a natural
isomorphism of $\kk$-vector spaces
\begin{equation*}
\tau_{_{M,N,W}}:Hom_V(M\oV N, W)\longrightarrow Hom_V(M,Vhom(N,W))
\end{equation*}
sending $f:M\oV N\to W$ to
\begin{equation*}
\big[\tau_{_{M,N,W}}(f)\big] (u)\, :\, v\mapsto f(u\oz v)
\end{equation*}
for $u\in M$ and $v\in N$. In more detail, fixing any two of
$M,N,W$, each $\tau_{_{M,N,W}}$ is a natural isomorphism:
\begin{align*}
    Hom_V(\,\square \oV N, W) & \longrightarrow Hom_V(\, \square \,,
    Vhom(N,W))\\
Hom_V(M \oV \square \, , W) & \longrightarrow Hom_V(M ,
    Vhom(\,\square ,W))\\
Hom_V(M \oV N , \square\,) & \longrightarrow Hom_V(M ,
    Vhom(N ,\square\, )).
\end{align*}
For example, if $f\in Hom_V(M,M'\,)$, then the following diagram
commutes:
\begin{equation}\label{natu}
\begin{CD}
 Hom_V(M\s ' \oV N,W) @> \tau_{_{M',N,W}} >> Hom_V(M\s ',Vhom(N,W))  \\
 @V (\,f\,\otimes \; 1_N)^{\, *} VV  @VV f^{\,*} V\\
 Hom_V(M \oV N,W) @>> \tau_{_{M,N,W}} > Hom_V(M,Vhom(N,W))
\end{CD}
\end{equation}
\end{theorem}

\begin{proof}
(cf. Theorem 7.2.1 in \cite{Li1}) Using Proposition
\ref{BBBB}, we have that $Hom_V(M\oV N,W)$ is linearly isomorphic
(in a natural way) to $V\hbox{-}\,Bilinear(M,N;W)$, practically by
definition. It remains to show that
\begin{equation*}
V\hbox{-}\,Bilinear(M,N;W)\simeq Hom_V(M,Vhom(N,W)).
\end{equation*}

Let $\phi\in Hom_V(M,Vhom(N,W))$. Then we define a $\kk$-bilinear
map $F^{\, \phi}_z:M\times N\to W((z))$ given by
\begin{equation*}
F^{\,\phi}_z(u,v)=[\phi(u)]_z(v) \qquad \hbox{ for } u\in M, v\in N.
\end{equation*}
By definition, we have
$$
F^{\,\phi}_z(du,v)=[\phi(du)]_z(v)=[d\,\phi(u)]_z(v)
=\frac{d}{dz}[\phi(u)]_z(v)=\frac{d}{dz}F^{\,\phi}_z(u,v)
$$
and
$$
F^{\,\phi}_z(u,dv)=[\phi(u)]_z(dv)=\Big(d-\frac{d}{dz}\Big)[\phi(u)]_z(v)
= \Big(d-\frac{d}{dz}\Big)F^{\,\phi}_z(u,v).
$$

\vskip .2cm

\noi Furthermore, for $a\in V, u\in M$ and $v\in N$, by the action on
$Vhom$ (\ref{vhom-action}), we have

\vskip -.2cm

\begin{align*}
F^{\,\phi}_x(a\z u,v)&=[\phi(a\z u)]_x(v)=[a\z
\phi(u)]_x(v)=a\zx([\phi(u)]_x(v))-[\phi(u)]_x(a\zx v-a\xz v)\\
&=a\zx(F^{\,\phi}_x(u,v))-F^{\,\phi}_x(u,a\zx v- a\xz v).
\end{align*}

\vskip .2cm

\noi Thus, for all $a\in V$ and $v\in N$, there exists $n\in \NN$
($n\geq N_{a,v}$) such that

\vskip -.1cm

\begin{equation*}
(z+x)^n F^{\,\phi}_x (a\z u,v)=(z+x)^n a\zx F^{\,\phi}_x(u,v), \qquad
\hbox{ for all } u\in M.
\end{equation*}

\vskip .3cm

\noi Using (b) in the definition of $Vhom$, for all $a\in V$ and $u\in
M$ there exists $k\in \NN$ such that

\vskip -.1cm

\begin{equation*}
(z-x)^k F^{\,\phi}_x(u,a\z v)=(z-x)^k a\z (F^{\,\phi}_x(u,v)) \qquad
\hbox{ for all } v\in N.
\end{equation*}

\vskip .3cm

\noi Therefore, using Definition \ref{bilin}.(3), we have that $F^{\,\phi}_x$
is a vertex bilinear map of type $(M,N;W)$, obtaining a linear map
from $Hom_V(M,Vhom(N,W))$ to $V\hbox{-}\,Bilinear(M,N;W)$.

Conversely, for any $F_z\in V\hbox{-}\,Bilinear(M,N;W)$, it is clear by
Definition \ref{bilin}.(3), that for each $u\in M$, we have
$F_z(u, \cdot)\in Vhom(N,W)$.
 Then we obtain a linear map $\phi^F:M\to Vhom(N,W)$ (compare it with
 $\tau_{_{M,N,W}}$) defined by
$$
[\phi^F(u)]_z(v)=F_z(u,v).
$$

\vskip .2cm

\noi For any $a\in V$ and $u\in M$, using  the associator formula
(\ref{Iassoc}) and the action on $Vhom$ (\ref{vhom-action}), we
get
$$
[\phi^F(a\z u)]_x(v)=F_x(a\z u,v)=a\zx F_x(u,v)-F_x(u, a\zx v-a\xz
v)=[a\z\phi^F(u)]_x(v).
$$

\vskip .2cm

\noi Therefore $\phi^F\in Hom_V(M,Vhom(N,W))$, proving the
isomorphism.

Now, we check that the maps $\tau_{_{M,N,W}}$ are natural in $M$
(see (\ref{natu})): for $\varphi\in Hom_V(M\s '\oV N,W)$, we have
\begin{align*}
    \Big[\big(\tau_{_{M,N,W}}\big[(\, f\otimes
    1_N)^*(\varphi)\big]\big)(u)\Big]_z(v) & =\big[(\, f\otimes
    1_N)^*(\varphi)\big](u\oz v)=\varphi\big((\, f\otimes 1_N)(u\oz
    v)\big)\\
    & =\varphi(f(u)\oz v),
\end{align*}
and
\begin{equation*}
\Big(\big[f^*(\tau_{_{M,N,W}}(\varphi))\big](u)\Big)_z(v)=\Big(\big[\tau_{_{M\s
',N,W}}(\varphi)\big](f(u))\Big)_z(v)=\varphi(f(u)\oz v)
\end{equation*}

\vskip .2cm

\noi proving naturality in $M$. One can check that the other cases
are natural by standard arguments.
\end{proof}

\

In the next part of the work, we follow in part \cite{DLM}.

\

\begin{definition} Let $M$ and $N$ be two $V$-$\,$modules. A {\it
right vertex homomorphism} from $M$ to $N$ is a liner map
$f_z:M\longrightarrow N((z))$ such that

\vskip .3cm

(a) $f_z(d \, u)=\frac{d}{dz}\, f_z(u)$ \ for all $u\in M$.

\vskip .3cm

(b) For any $a\in V$, there exists $k\in\mathbb{N}$ such that
\begin{equation} \label{right-b}
(z+x)^k\ a\zx (f_x(u))=(z+x)^k \ f_x(a\z u)\qquad \textrm{for all
} u\in M.
\end{equation}

\end{definition}

\vskip .3cm

Observe that  in \cite{DLM}  Definition 2.11, they give a similar
notion, but the integer $k$ in (b) also depends on $u\in M$, which
is wrong.

We denote by $Vhom^r(M,N)$ the space of all right vertex
homomorphisms from $M$ to $N$. Define $d$ on $Vhom^r(M,N)$ given
by $(d\, f)_z(u)=\big(d-\frac{d}{dz}\big)(f_z(u))$, for $u\in M$.
Observe that if $V$ is an associative commutative algebra with
unit and $d=0$, and $M$ and $N$ are modules over it with the
corresponding derivations equal to zero, then (using (a) in the
definition) any right vertex homomorphism must be independent of
$z$ and $Vhom^r(M,N)=Hom_V(M,N)$, the usual homomorphisms of
modules over an associative commutative algebra.

One of the motivation for the definition of right vertex
homomorphism is the following: let $g_z:M\times N\to W((z))$ be a
vertex bilinear map. Then for any fixed  $v\in N$, the map
$f_z(u)=g_z(u,v)$ is a right vertex homomorphism from $M$ to $W$.

Now, we define an action of $V$ on $Vhom^r(M,N)$ as follows:
\begin{align*}
    (a\rz f)_x(u)=a\z(f_x(u))-f_x(a\zmx u - a\mxz u)
\end{align*}
which can also be rewritten as
\begin{align*}
(a\rz f)_x(u)=a\z(f_x(u))-\mathrm{Res}_{\,y} \, \,
\de(x+y,z)f_w(a\y u),
\end{align*}
for $a\in V$ and $u\in M$.

Similarly, we can prove (see Proposition \ref{cheta} (e) and (f)):

\

\begin{proposition} Let $a\in V, f\in Vhom^r(M,N)$ and $u\in M$. Then we
have the following properties:

\vskip .3cm

(a)  For any integer $k$ such that $k\geq N_{a,u}$, we have

\begin{equation*}
(z-x)^k \ (a\rz f)_x (u)= (z-x)^k \ a\z(f_x(u)) \ \in N((z,x)).
\end{equation*}

\vskip .2cm

(b) For any integer $k$ such that $k\geq N_{a,u}$, an alternative
definition for the action of $V$ in $Vhom^r(M,N)$ is given by
\begin{equation*}
\ (a\rz f)_x(u)= (-x+z)^{-k} \big((z-x)^k \ a\z(f_x(u)) \big).
\end{equation*}
similar to the ring theoretical case.
\end{proposition}

\

Now, we define $T$ as the operator in $Hom_{\kk}(M,N((z)))$ given
by

\vskip -.1cm

\begin{equation*}
\big[T(f)\big]_z(u)=e^{zd} \, f_{-z}(u),
\end{equation*}

\vskip .2cm

\noi for any $f\in Hom_{\kk}(M,N((z)))$ and $u\in M$. Then we have the
following result (cf. Proposition 2.12 in \cite{DLM}).

\

\begin{proposition}
 Let $f\in
Hom_{\kk}(M,N((z)))$. Then $f\in Vhom (M,N)$ if and only if
$T(f)\in Vhom^r(M,N)$.
\end{proposition}

\begin{proof}
For $u\in M$, by definition we have
\begin{align*}
     \frac{d}{dz}[T(f)]_z(u)&=d\, e^{zd} \, f_{-z}(u) + e^{zd}\,
     \frac{d}{dz}\big(f_{-z}(u)\big)\\
     &=e^{zd}\Big(d \, f_{-z}(u) -f_{-z}(du) +\,
     \frac{d}{dz}\big(f_{-z}(u)\big)\Big)+[T(f)]_z(du).
\end{align*}
Then $\frac{d}{dz}[T(f)]_z(u)=[T(f)]_z(du)$ if and only if
$\big(d-\frac{d}{dz}\big)f_z(u)=f_z(du)$.

\

For any $a\in V$, suppose that there exists $k\in \NN$ such that
\begin{equation*}
(z+x)^k\ a\zx ([T(f)]_x(u))=(z+x)^k \ [T(f)]_x(a\z u)\qquad
\textrm{for all } u\in M.
\end{equation*}
Then
\begin{equation*}
(z+x)^k\ e^{-xd}\  a\zx ([T(f)]_x(u))=(z+x)^k \ e^{-xd}\
[T(f)]_x(a\z u).
\end{equation*}
Thus
\begin{equation*}
(z+x)^k\  a\z (e^{-xd}\ [T(f)]_x(u))=(z+x)^k \ e^{-xd}\
[T(f)]_x(a\z u).
\end{equation*}
That is,
\begin{equation*}
(z+x)^k\  a\z (f_{-x}(u))=(z+x)^k \ f_{-x}(a\z u).
\end{equation*}

\vskip .2cm

\noi Since every step in the proof can be reversed, $f$ satisfies
(\ref{left-b}) if and only if $T(f)$ satisfies (\ref{right-b}).
\end{proof}

\

\begin{proposition}
{\rm (cf. Proposition 2.13 in \cite{DLM})} For any $a\in V$ and $f\in
Vhom(M,N)$, we have
\begin{equation*}
T(a\z f)=a\rz [T(f)].
\end{equation*}
\end{proposition}

\begin{proof}
By definition we have

\begin{align*}
    [T(a\z f)]_x(u) &= e^{xd}\big( a\zmx(f_{-x}(u))-f_{-x}(a\zmx u
    - a\mxz u)\big)\\
&= a\z(e^{xd}f_{-x}(u))-e^{xd}f_{-x}(a\zmx u
    - a\mxz u)\\
&= a\z([T(f)]_{x}(u))-[T(f)]_{x}(a\zmx u
    - a\mxz u)\\
&= (a\rz [T(f)])_x(u).
\end{align*}
\end{proof}

\noi Then we obtain

\begin{proposition}
$Vhom^r(M,N)$ is a $V$-$\,$module and the linear map $T$ is a
$V$-$\,$isomorphism.
\end{proposition}

\

Let $M,M\s '$ and $N$ be $V$-$\,$modules. Given $f\in Hom_V(M\s ',M)$,
we can define an induced map
\begin{equation*}
f_*: Vhom^r(N,M\s ')\to Vhom^r(N,M)
\end{equation*}
given by $[f_*(\psi)]_z(v)=f(\psi_z(v))$ for all $v\in N$.
Similarly, we define
\begin{equation*}
f^*: Vhom^r(M,N)\to Vhom^r(M\s ',N)
\end{equation*}
given by $[f^*(\phi)]_z(u')=\phi_z(f(u'))$ for all $u'\in M\s '$.
A simple computation shows that $f_*(\psi)\in Vhom^r(N,M)$ and
$f^*(\phi)\in Vhom^r(M\s ', N)$.

Observe that we use the same notation of $f_*$ and $f^*$ for
$Vhom^r$, $Vhom$ and $Hom_V$. The proof of the following result is
similar to the proof of Proposition \ref{54}.

\begin{proposition}
(a) The maps $f_*$ and $f^*$ are $V$-$\,$homomorphisms.

\vskip .2cm

\noi (b) $Vhom^r(N,\square)$ is a covariant functor from
\textsf{Mod}$_V$ to \textsf{Mod}$_V$.

\vskip .2cm

\noi (c) $Vhom^r(\square,N)$ is a contravariant functor from
\textsf{Mod}$_V$ to \textsf{Mod}$_V$.
\end{proposition}

Now, we shall present some properties of $Vhom^r$.

\begin{proposition}
Let $M$ be a $V$-$\,$module. Then $Vhom^r(V,M)$ is isomorphic to $M$
as a $V$-$\,$module. The isomorphism is given by
$\Psi_M:Vhom^r(V,M)\to M$ with $\Psi_M(f)=\hbox{Res}_{\,z}\ z^{-1} \,
f_z(\vac)$. Moreover, if $g\in Hom_V(M,N)$, then the following
diagram commutes:

$$\begin{CD}
 Vhom^r(V,M) @>\Psi_M>> M \\
 @Vg_*VV  @VVgV\\
 Vhom^r(V,N) @>>\Psi_N> N
\end{CD}$$

\noi proving that $\Psi$ is a natural isomorphism from
$Vhom^r(V,\square)$ to the identity functor on \textsf{Mod}$_V$.
\end{proposition}

\begin{proof}
We define a map
\begin{equation*}
\Phi:M\longrightarrow Vhom(V,M)
\end{equation*}
given by $[\Phi(u)]_z(a)=a\z u$, for $u\in M$. We should check
that $[\Phi(u)]_z\in Vhom^r(V,M)$. By definition, we have
\begin{align*}
    [\Phi(u)]_z(da)&=(da)\z u=\frac{d}{dz}(a\z
    u) =\frac{d}{dz}[\Phi(u)]_z(a).
\end{align*}
Using associativity, for any $b\in V$, there is $k\in \mathbb{N}$
depending only on $b$ and $u$ (see Remark \ref{module}) such that
\begin{equation*}
(z+x)^k b\zx (a\x u)=(z+x)^k (b\z a)\x u
\end{equation*}
for all $a\in V$. Thus,
\begin{align*}
    (z+x)^k b\zx \big([\Phi(u)]_{x}(a)\big)=(z+x)^k
    \big ([\Phi(u)]_{x}(b\z
    a)\big)
\end{align*}
for all $a\in V$, proving that $[\Phi(u)]_z\in Vhom^r(V,M)$. It is
easy to see using the commutator formula for modules
(\ref{commutator-m}) that the linear map $\Phi$ from $M$ to
$Vhom^r(V,M)$ is a $V$-$\,$homomorphism.

Conversely, given any $f_z\in Vhom^r(V,M)$, we will prove that
$f=\Phi(\Psi_M(f))$. If we write $f_z=\sum_{n\in \mathbb{Z}} f_n
z^{-n-1}$, then by definition $f_n(da)=n f_{n-1}(a)$. In
particular, $n f_{n-1}(\mathbf{1})=0$ for all $n\in \ZZ$.
Therefore, $u:=f_{-1}(\mathbf{1})=f_z(\vac)\in M$. Now, for any
$a\in V$ there exists $k\in \mathbb{N}$ such that $(z+x)^k f_z(a\x
b)=(z+x)^k a\xz(f_z(b))$ for all $b\in V$. Then
\begin{align*}
    z^k (a\z u)& =  (x+z)^k \, a\xz (f_z(\mathbf{1}))\, _{|_{x=0}}=
     (x+z)^k f_z( a\x \mathbf{1}))\, _{|_{x=0}}= (x+z)^k  f_z(e^{xd}a)
     \, _{|_{x=0}}\\
& = (x+z)^k e^{x\frac{d}{dz}} f_z(a)\, _{|_{x=0}}=(x+z)^k
f_{x+z}(a)=z^k f_z(a).
\end{align*}
Thus, $f=\Phi(\Psi_M(f))$. Therefore $Vhom^r(V,M)$ is isomorphic
to $M$ as a $V$-$\,$module. Finally, we have
\begin{equation*}
\Psi_N(g_*(\psi))=\hbox{Res}_z \, z^{-1} \, [g_*(\psi)]_z(\vac)=
\hbox{Res}_z \, z^{-1} \, g(\psi_z(\vac))=g(\Psi_M(\psi)),
\end{equation*}
finishing the proof.
\end{proof}

\

The following result could be thought as a motivation for the
definition of $Vhom^r$ and the action of $V$ on it. The proof is
similar to the proof of Theorem \ref{TT}

\

\begin{theorem}\label{TT22} {\rm (Adjoint isomorphism, second version)}
Let $M,N$ and $W$ be $V$-$\,$modules. Then there is a natural
isomorphism of $\kk$-vector spaces
\begin{equation*}
\tau\, '_{_{M,N,W}}:Hom_V(N\oV M, W)\longrightarrow
Hom_V(M,Vhom^r(N,W))
\end{equation*}
sending $f:N\oV M\to W$ to
\begin{equation*}
\big[\tau\, '_{_{M,N,W}}(f)\big] (u)\, :\, v\mapsto f(v\oz u)
\end{equation*}

\vskip .2cm

\noi for $u\in M$ and $v\in N$. That is, fixing any two of $M,N,W$,
each $\tau\, '_{_{M,N,W}}$ is a natural isomorphism.
\end{theorem}

\

If we consider all the results together, we obtain:

\

\begin{corollary} \label{TTTTT}
For any $V$-$\,$modules $M,N$ and $W$, we have the
following linear isomorphisms
\begin{align*}
\MNW &\simeq V\hbox{-}\,Bilinear(M,N;W)\simeq Hom_V(M\oV N,W)\\
&\simeq Hom_V(M,Vhom(N,W))\simeq Hom_V(N,Vhom^r(M,W)).
\end{align*}
\end{corollary}

\

%%%%%%%%%%%%%%%%%%%%%%%%%%
\section{Properties of the tensor product}\lbb{second}
%%%%%%%%%%%%%%%%%%%%%%%%%%

\

We start this section by proving the commutativity of the tensor
product. Then we prove the associativity of the tensor product
under certain (algebraic and natural) necessary and sufficient
conditions.

The {\it transpose of a vertex bilinear map (or an intertwining
operator)} $F_z$ is defined as
\begin{equation}\label{transpose}
    (F_z)^t(u,v):=e^{zd}F_{-z}(v,u)
\end{equation}
and it is a vertex bilinear map (or an intertwining operator) of
the corresponding type. We have the following result (cf.
Proposition 5.1.6, Proposition 5.1.7 and Theorem 6.2.3 in
\cite{Li1}):

 \

\begin{theorem}
Let $M$ and $N$ be  $V$-$\,$modules.

\vskip .3cm

\noi (a) There is a natural $V$-$\,$isomorphism
\begin{equation*}
\varphi_{_M}:V\oV M\longrightarrow M
\end{equation*}
where $\varphi_{_M}(a\oz u)=a\Mz u$ for $a\in V$ and $u\in M$, and
  $(M,\Mz)$ is a tensor product of $(V,M)$. Symmetrically,
$(M,(\Mz)^t)$ is a tensor product of $(M,V)$.

\vskip .3cm

\noi(b) If $(M\oV N, \, \ozz\,)$ is a tensor product of the pair
$(M,N)$, then $(M\oV N, \, (\ozz)^t\,)$ is a tensor product of the
pair $(N,M)$.

\vskip .3cm

\noi(c) The map
\begin{align*}
\tau:\, &M\oV N\longrightarrow N\oV M\\
&\,\ u\oz v \longmapsto e^{zd}(v\omz u)
\end{align*}

\noi is a natural $V$-$\,$isomorphism in the sense that the following
diagram commutes

$$\begin{CD}
 M\oV N @>\tau>> N\oV M \\
 @V f\,\otimes \,g VV  @VVg\,\otimes \,fV\\
 M\s '\oV N\s ' @>>\tau\s '> N\s ' \oV M\s '
\end{CD}$$

\end{theorem}

\

\begin{proof}
(a) Since $\Mz$ and $(\Mz)^t$ satisfy the Jacobi identity, then
they are intertwining operators of the corresponding type. Let $W$
be any $V$-$\,$module, and let $I_z$ be an intertwining operator of
type $\binom{W}{V,M}$. Since
$\frac{d}{dz}I_z(\mathbf{1},u)=I_z(d\,\mathbf{1},u)=0$, then
$I_z(\mathbf{1},u)$ is independent of $z$. Using the commutator
formula (\ref{Icommut}), we obtain that $I_z(\mathbf{1},\, \cdot\,
)$ commutes with any operator $a\x$ for $a\in V$, namely
\begin{equation}\label{7777}
a\Wx I_z(\mathbf{1},u)=I_z(\mathbf{1},a\Mx u).
\end{equation}
Then, $\psi(u):=I_z(\mathbf{1}, u )$ is a $V$-$\,$homomorphism from
$M$ to $W$ and it satisfies the universal property if
$I_z(a,u)=\psi(a\Mz u)=I_z(\mathbf{1}, a\Mz u)$. But, using the
iterated formula (\ref{Ittt}) and (\ref{7777}), we have
\begin{align*}
     I_z(a,u)&= \mathrm{Res}\,_x \, \, x^{-1} \,
     I_z(a\x\mathbf{1},u)\\
     &=\mathrm{Res}\,_x \, \mathrm{Res}\,_y \, x^{-1} \,
     \big[\de(y-z,x) a\Wy I_z(\mathbf{1},u)-\de(-z+y,x) I_z(\mathbf{1}, a\My u)\big]\\
&=\mathrm{Res}\,_y \, [(y-z)^{-1} -(-z+y)^{-1}]\, I_z(\mathbf{1},
a\My u) = \mathrm{Res}\,_y \, \de(y,z)\, I_z(\mathbf{1}, a\My u)
=I_z(\mathbf{1},a\Mz u)
\end{align*}
for $a\in V, u\in M$, proving that $(M,\Mz)$ is a tensor product
of $(V,M)$. Let $F_z:V\times M\to M((z))$ be the $V$-$\,$bilinear map
defined by $F_z(a,u)=a\Mz u$. Then, it induces a $V$-$\,$homomorphism
 $\varphi_{_M}:V\oV M\longrightarrow M$
with $\varphi_{_M}(a\oz u)=a\Mz u$. By the uniqueness of the
tensor product, $\varphi_{_M}$ is a $V$-$\,$isomorphism. Naturality is
proved by showing commutativity of the following diagram for $f\in
Hom_V(M,N)$

$$\begin{CD}
 V\oV M @>\varphi_{_M}>> M \\
 @V 1\otimes f VV  @VVfV\\
 V\oV N @>>\varphi_{_N}> N
\end{CD}$$

\vskip .2cm

\noi It suffices to check the maps on generators:
\begin{equation*}
f\big(\varphi_{_M}(a\oz u)\big)=f(a\z u)=a\z
f(u)=\varphi_{_N}(a\oz f(u))=\varphi_{_N}\big((1\otimes f)(a\oz
u)\big).
\end{equation*}

\noi Similarly, the universal property can be proved for $(M,(\Mz)^t)$, 
or it follows by using (b).

\vskip .2cm

(b) Let $W$ be any $V$-$\,$module and let $I_z\in V\hbox{-}\,Bilinear(N,M;W)$.
It is easy to see that there is a $V$-$\,$homomorphism $\psi$ from
$M\oV N$ to $W$ such that $(I_z)^t=\psi\circ\ozz$ if and only if
$I_z=\psi\circ(\ozz)^t$, finishing (b).

\vskip .2cm

(c) Let $F_z:M\times N\to (N\oV M)((z))$ be given by
$F_z(u,v)=e^{zd}(v\omz u)$ for $u\in M$ and $v\in N$. It easy to
see that $F_z$ is a vertex bilinear map, and so there is a unique
$V$-$\,$homomorphism $\tau:M\oV N\rightarrow N\oV M$ with  $\tau( u\oz
v)= e^{zd}(v\omz u)$. Similarly, interchanging the roles of $M$
and $N$, we obtain a $V$-$\,$homomorphism $\tilde\tau : N\oV
M\rightarrow M\oV N $ with $\tilde\tau( v\oz u )= e^{zd}(u\omz
v)$. Both composite of these maps are obviously identity maps, and
so $\tau$ is a $V$-$\,$isomorphism. The proof of naturality is the
following:
\begin{equation*}
\tau \s '\big((f\otimes g)(u\oz v)\big)=\tau\s '\big( f(u)\oz
g(v)\big)=e^{zd}(g(v)\omz f(u))=(g\otimes f)(\tau(u\oz v)).
\end{equation*}
\end{proof}

Now, we prove the associativity of the tensor product under
certain (algebraic and natural) necessary and sufficient
condition. We shall use the universal property of the tensor
product instead of the explicit construction, following in part
the ideas in \cite{DLM} but using algebraic methods instead of
analytic ones (by the way, \cite{DLM} is full of typos that can be
easily fixed). We found a very natural condition that
simplify Huang's convergence assumptions.

Let $M,N$ and $W$ be three $V$-$\,$modules. We assume for the rest of
this section the existence of  the following  tensor products
$M\oV N, (M\oV N)\oV W, N\oV W$, and $M\oV (N\oV W)$.

The next two propositions immediately follow from the universal
property of tensor product in Definition \ref{tensor2}.

\begin{proposition}\label{3.1}
The tensor products $N\oV W$ and $M\oV(N\oV W)$ (and the
corresponding vertex bilinear maps) satisfy the following
universal property: For any $V$-$\,$modules $Y$ and $Z$, and any
vertex bilinear maps $H\in V\hbox{-}\,Bilinear(N,W;Z)$ and $I\in
V\hbox{-}\,Bilinear(M,Z;Y)$, there exists a unique $V$-$\,$homomorphism
$g:M\oV(N\oV W)\to Y$ such that
\begin{equation*}
g(u\oz(v\oy w))=I_z(u,H_y(v,w))
\end{equation*}
for all $u\in M,v\in N$ and $w\in W$. This universal property
characterizes $M\oV(N\oV W)$ uniquely.
\end{proposition}

Similarly, we have

\begin{proposition}\label{3.2}
The tensor products $M\oV N$ and $(M\oV N)\oV W$ (and the
corresponding vertex bilinear maps) satisfy the following
universal property: For any $V$-$\,$modules $Y$ and $Z$, and any
vertex bilinear maps $F\in V\hbox{-}\,Bilinear(M,N;Z)$ and $G\in
V\hbox{-}\,Bilinear(Z,W;Y)$, there exists a unique $V$-$\,$homomorphism
$f:(M\oV N)\oV W\to Y$ such that
\begin{equation*}
f((u\oz v)\oy w)=G_y(F_z(u,v),w)
\end{equation*}
for all $u\in M,v\in N$ and $w\in W$. This universal property
characterizes $(M\oV N)\oV W$ uniquely.
\end{proposition}

\vskip .2cm

Now, we can state the main result of this section.

\begin{theorem}
Let $M,N$ and $W$ be three $V$-$\,$modules. Assume  the existence of
the following  tensor products $M\oV N, (M\oV N)\oV W, N\oV W$,
and $M\oV (N\oV W)$. Then: there exists a unique $V$-$\,$isomorphism
\begin{equation*}
f:(M\oV N)\oV W\to M\oV (N\oV W)
\end{equation*}
such that (coefficient-wise)
\begin{equation*}
f\big((u\ox v)\oy w \big)= u \oxy (v\oy w)
\end{equation*}

\

\noi for all $u\in M, v\in N \hbox{ and } w\in W$ if and only if
the following expressions exist and satisfy:

\vskip .4cm

\noi (1) for any $u\in M$, $v\in N$ and $w\in W$

\begin{equation}\label{cond1}
 u\oxy (v\oy w)\in \Big[M\oV (N\oV W)\Big]((y))((x))
\end{equation}

\vskip .3cm

\noi where the number of negative powers of $x$ depends only on
$u$ and $v$,
and

\vskip .4cm

\noi (2) for any $u\in M$, $v\in N$ and $w\in W$

\begin{equation}\label{cond2}
 (u\oxmy v)\oy w\in \Big[(M\oV N)\oV W\Big]((x))((y))
\end{equation}

\vskip .3cm

\noi where the number of negative powers of $y$ depends only on
$v$ and $w$.
\end{theorem}

\

Observe that  the expressions always exist, except for $(u\oxmy
v )\oy w$. We also have that
\begin{equation}\label{cond 2}
(u\ox v)\oy w \, \in \Big[(M\oV N)\oV W\Big]((y))((x)),
\end{equation}
where the finite number of negative powers in $x$ depends only on $u\in M$ and
$v\in N$, and
\begin{equation} \label{cond 1}
u\ox(v\oy w) \, \in \Big[M\oV(N\oV W)\Big]((x))((y)),
\end{equation}
where the finite number of negative powers in $y$ depends only on $v\in N$ and
$w\in W$.

\

A simple argument that prove the "only if" part of the theorem, is
the following: suppose that there exists the $V$-$\,$isomorphism $f$.
Then, by considering the left hand side of $f((u\ox
v)\oy w )= u \oxy (v\oy w)$, and using (\ref{cond 2}),
 we have that the right hand side must satisfy (\ref{cond1}).
 Since $f$ is an isomorphism, then
it is clear that $g:M\oV (N\oV W)\to (M\oV N)\oV W$ with
 $g(u\ox (v\oy w ))= (u \oxmy v)\oy w$ is well defined and it is
the inverse of $f$, and using (\ref{cond 1}), we have that (\ref{cond2})
 must be satisfied.

By using the vertex bilinear properties (that are preserved by
homomorphism), one can see why we should send $(u\ox v)\oy w $ to
$ u \oxy (v\oy w)$ instead of $u \ox (v\oy w)$ when we construct
the isomorphism.

\

The rest of this section is devoted to the proof of the "if" part
of the associativity theorem. The sketch of proof or the basic
idea of the proof is the following: first, we define $F:M\times N
\to Vhom(W, M\oV (N\oV W))((z))$ given by
\begin{equation*}
\Big[F_z(u,v)\Big]_y(w)=u\ozy (v\oy w),
\end{equation*}
and we prove that $F$ is a vertex bilinear map of type $(M,N;
Vhom(W, M\oV (N\oV W)))$. Then we take $G$ as the natural vertex
bilinear map of type $(Vhom(W, M\oV (N\oV W)), W; M\oV (N\oV W))$
given by $G_y(q,w)=q_{\, y}(w)$ for $q\in Vhom(W, M\oV (N\oV W))$ and
$w\in W$. After that, we apply Proposition \ref{3.2} to the vertex
bilinear maps $F$ and $G$ to get a $V$-$\,$homomorphism $f:(M\oV N)\oV
W\to M\oV (N\oV W)$ such that $f((u\oz v)\oy w)=u\ozy (v\oy w)$.
Similarly, we define $H:N\times W \to Vhom^r(M, (M\oV
 N)\oV W)((y))$ given by
\begin{equation*}
\Big[H_y(v,w)\Big]_z(u)=(u\ozmy v)\oy w,
\end{equation*}
and we prove that $H$ is a vertex bilinear map of type $(N,W;
Vhom^r(M, (M\oV N)\oV W))$. Then we take $I$ as the natural vertex
bilinear map of type $(M,Vhom^r(M, (M\oV N)\oV W); (M\oV N)\oV W)$
given by $I_y(u,h)=h_y(u)$ for $h\in Vhom^r(M, (M\oV N)\oV W)$ and
$u\in M$. After that, we apply Proposition \ref{3.1} to the vertex
bilinear maps $H$ and $I$ to get a $V$-$\,$homomorphism $g:M\oV (N\oV
W)\to (M\oV N)\oV W$ such that $g(u\oz (v\oy w))=(u\ozmy v)\oy w$,
obtaining the desired isomorphisms $f$ and $g$. This is the
algebraic and non-graded version of the proof in \cite{DLM}.

\

For the rest of this section we assume that all the tensor
products exist, and we assume the conditions
(\ref{cond1})-(\ref{cond2}) about the Laurent series in the
statement of the theorem.

For any fixed $u\in M$ and $v\in N$, we define a linear map
\begin{equation*}
    \psi(u,v;z,y)\,:\, W\to \Big[(M\oV N)\oV W\Big]((y))((z))
\end{equation*}
as
\begin{equation*}
\psi(u,v\, ;z,y)(w):=u\ozy (v\oy w)
\end{equation*}

\noi for any $w\in W$. Using (\ref{cond1}),  we have
\begin{equation*}
\psi(u,v\,;z,y)=\sum_{n\in \ZZ} \, \psi_n(u,v\, ;y)\, z^{-n-1}
\end{equation*}
with $\psi_k(u,v\, ;y)\equiv 0$ for $k$ sufficiently large and
\begin{equation*}
\psi_n(u,v\, ;y):W\to \Big[M\oV (N\oV W)\Big]((y)).
\end{equation*}

\vskip .4cm

\begin{proposition}
For any $u\in M, v\in N$ and $n\in \ZZ$, we have

\begin{equation*}
\psi_n(u,v\, ; y)\in Vhom(W, M\oV (N\oV W)).
\end{equation*}

\end{proposition}

\vskip .2cm

\begin{proof}
For  $u\in M, v\in N, w\in W$ and $n\in \ZZ$, we have
\begin{align*}
    \psi(u,v\, ;z,y)(d\, w) & = u\ozy(v\oy
    dw)=u\ozy\Big(d-\frac{d}{dy}\Big)(v\oy w)=
    \Big(d-\frac{d}{dy}\Big)\big(u\ozy(v\oy w)\big)\\
    & =\Big(d-\frac{d}{dy}\Big)\, \psi(u,v\, ;z,y)(w).
\end{align*}
Then $\psi_n(u,v\, ;y)$ satisfies $(a)$ in Definition \ref{Vhom}.
Since $\psi_m(u,v\, ;y)=0$ for all $m\geq N$, for some $N\in\NN$,
then
\begin{equation}\label{3.}
\hbox{Res}_{\, z} \ z^m\, f(y,z)\, \psi(u,v\, ;z,y)=0
\end{equation}

\vskip .3cm

\noi for any $f\in \kk[\,y,z\,]$. Now, using the Definition
\ref{bilin}.(3) where the exponents depend on certain variables,
we have that for any $a\in V$ and $n\in \ZZ$, we take  $k\in \NN$
such that $k+n\geq N$ and

%vskip .01cm
%
\begin{align}\label{1.}
    (x-z)^k\, \, a\x\big(u\oz(v\oy w)\big) & = (x-z)^k \, \, u\oz
    \big( a\x
    (v\oy w)\big) \qquad \hbox{ for all } v\in N, w\in W,\\
\label{2.}
    (x-y)^k\, \, a\x (v\oy w) & = (x-y)^k  \, \,
    (v\oy a\x w) \qquad \qquad \ \ \hbox{ for all }  w\in W.
\end{align}

\noi Then, using (\ref{3.}), (\ref{1.}) and (\ref{2.}), we have

\begin{align*}
(x-y)^{3k}\,\ & a\x \psi_n(u,v\, ; y)(w)  = \hbox{Res}_{\,z}\, z^n
(x-y)^{3k}\ a\x\psi(u,v\, ;z,y)(w)\\
& = \sum_{i=0}^{2k}\ \binom{2k}{i} \ \hbox{Res}_{\,z}\ (x-z-y)^i
z^{n+2k-i} (x-y)^k \ a\x \psi (u,v\, ;z,y)(w)\\
& = \sum_{i=k}^{2k}\ \binom{2k}{i} \ \hbox{Res}_{\,z}\ (x-z-y)^i
z^{n+2k-i} (x-y)^k \ a\x \big(u\ozy(v\oy w)\big)\\
& = \sum_{i=k}^{2k}\ \binom{2k}{i} \ \hbox{Res}_{\,z}\ (x-z-y)^i
z^{n+2k-i} (x-y)^k \  \big(u\ozy(v\oy a\x w)\big)\\
& = \sum_{i=k}^{2k}\ \binom{2k}{i} \ \hbox{Res}_{\,z}\ (x-z-y)^i
z^{n+2k-i} (x-y)^k \ \psi(u,v\, ; z,y)(a\x w)\\
& = \sum_{i=0}^{2k}\ \binom{2k}{i} \ \hbox{Res}_{\,z}\ (x-z-y)^i
z^{n+2k-i} (x-y)^k \ \psi(u,v\, ; z,y)(a\x w)\\
& = \hbox{Res}_{\,z}\, z^n
(x-y)^{3k}\ \psi(u,v\, ;z,y)( a\x w)\\
& = (x-y)^{3k}\,\  \psi_n(u,v\, ; y)(a\x w)\qquad \qquad \hbox{
for all } w\in W.
\end{align*}
Thus $\psi_n(u,v\, ;y)\in Vhom(W, M\oV (N\oV W))$ for any $n\in
\ZZ, u\in M, v\in N$.
\end{proof}

Now, using the previous proposition, we define the $\kk$-bilinear
map
$$
F:M\times N \to Vhom(W, M\oV (N\oV W))((z))
$$
given by
\begin{equation*}
\Big[F_z(u,v)\Big]_y(w)=\psi(u,v\, ;z,y)(w)=u\ozy (v\oy
w)=\sum_{n\in \ZZ} \psi_n(u,v\, ;y)(w)\, z^{-n-1},
\end{equation*}
for any $u\in M,v\in N$ and $w\in W$.

\

\begin{proposition}\label{F}
$F_z$ is a vertex bilinear map of type $(M,N; Vhom(W, M\oV (N\oV
W)))$.
\end{proposition}

\begin{proof}
For $u\in M,v\in V$ and $w\in W$, we have
\begin{align*}
\Big[F_z(du,v)\Big]_y(w) & =(du)\ozy (v \oy w)=\frac{d}{dx}
\big(u\ox(v\oy w)\big)_{|_{x=z+y}} = \frac{d}{dz} (u\ozy (v\oy
w))
\\
& = \Big[\frac{d}{dz} \, F_z(u,v)\Big]_y (w),
\end{align*}
and
\begin{align*}
\Big[\Big(d-\frac{d}{dz}\Big)\, F_z(u,v)\Big]_y(w) & =
\Big(\frac{d}{dy}-\frac{d}{dz}\Big)\,
\Big(\Big[F_z(u,v)\Big]_y(w)\Big) =
\Big(\frac{d}{dy}-\frac{d}{dz}\Big)\,\big(u\ozy (v\oy w)\big ) =\\
& = u\ozy\, \frac{d}{dy}(v\oy w) = u\ozy (dv \oy
w)=\Big[F_z(u,dv)\Big]_y(w).
\end{align*}

\vskip .3cm

\noi Then $F_z$ satisfies the properties with respect to $d$. Now,
we prove the commutator formula, from which we obtain (\ref{BB}).
For $a\in V, u\in M, v\in N$ and $w\in W$, using the action on
$Vhom$, we have
\begin{align*}
    \Big[a\x F_z(u,v)\Big]_y(w) & = a\xxy
    \Big(\big[F_z(u,v)\big]_y(w)\Big)-\big[F_z(u,v)\big]_y (a\xxy w
    - a\yx w)\\
    & = a\xxy \big(u\ozy (v\y w)\big)-u\ozy\big(v\oy (a\xxy w- a\yx
    w)\big)
\end{align*}
and
\begin{equation*}
\Big[F_z(u, a\x v)\Big]_y (w)= u\ozy (a\x v\oy w).
\end{equation*}
Then, using associator formula (\ref{Iassoc}) and the action on
the tensor product, we have
\begin{align*}
\Big[a\x F_z(u,v)\Big]_y(w) & - \Big[F_z(u, a\x v)\Big]_y (w) =\\
& = a\xxy \big(u\ozy (v\oy w)\big)-u\ozy\big(a\x v\oy w + v\oy
(a\xxy
w- a\yx w)\big)\\
& =a\xxy \big(u\ozy (v\oy w)\big)-u\ozy a\xxy (v\oy w) \\
& =(a\xmz u - a\mzx u )\ozy (v\oy w)\\
& =\big[F_z(a\xmz u - a\mzx u,v)\big]_y (w)
\end{align*}
and this  gives rise to the commutator formula (\ref{Icommut})
which implies locality or (\ref{BB}). Next, we prove associativity
or (\ref{AA}). For $a\in V$ and $v\in N$, we take $k\in \NN$
(depending on $a$ and $v$) such that
\begin{equation}\label{st}
    (x-y)^k \, a\x (v\oy w)=(x-y)^k \, v\oy a\x w, \qquad \hbox{
    for all } w\in W.
\end{equation}
Now, using the definition of the action on $Vhom$, the associator
formula (\ref{Iassoc}) and (\ref{st}), we have
\begin{align*}
    &(x+z)^k  \Big(\big[F_z(a\x u,v)\big]_y(w) - \big[a\xz\,
    F_z(u,v)\big]_y(w)\Big) = \\
    & =(x+z)^k \Big((a\x u)\ozy(v\oy w)- a\xzy\big(u\ozy (v\oy
    w)\big)+ \hbox{Res}_{\,t}\ \delta(-y+t,x+z) \ u\ozy(v\oy a\t
    w)\Big)\\
    & = (x+z)^k \Big(\hbox{Res}_{\,t}\ \Big[\delta(-y+t,x+z) \ u\ozy(v\oy a\t
    w) - \delta(-z-y+t,x)\ u\ozy a\t (v\oy w)\Big]\Big)\\
    & = \hbox{Res}_{\,t}\ \Big[(x+z)^k\,\delta(-y+t,x+z) \ u\ozy(v\oy a\t
    w)\Big]\\
    & \ \ \ \ - \hbox{Res}_{\,t}\ \Big[
     \Big(\delta(x+z,-y+t)-\delta(z+x,-y+t)\Big)(-y+t)^k\
      u\ozy a\t (v\oy w)\Big]\\
    & =(x+z)^k\ \hbox{Res}_{\,t}\ \Big[
       \delta(z+x,-y+t)\
      u\ozy  (v\oy a\t w)\Big]\\
    & =(x+z)^k\ \hbox{Res}_{\,t}\ \Big[
       \delta(z+x+y,t)-\delta(y+z+x,t)\Big]\
      u\ozy  (v\oy a\t w).
\end{align*}
Therefore, if we take $l(=N_{a,w})$ such that $x^l a\x w\in
W[[x]]$, then
\begin{equation}\label{matadora}
    (x+z)^k (y+z+x)^l \big[F_z(a\x u,v)\big]_y(w) =
    (x+z)^k (y+z+x)^l  \big[a\xz\, F_z(u,v)\big]_y(w).
\end{equation}
Notice that $(y+x+z)^{-l}=(y+z+x)^{-l}$ and that we are allowed to
multiply the right-hand side of (\ref{matadora}) by $(y+x+z)^{-l}$
and to multiply the left-hand side by $(y+z+x)^{-l}$. Then
multiplying both sides by $(y+x+z)^{-l}$ we obtain
\begin{equation*}
(x+z)^k  \big[F_z(a\x u,v)\big]_y(w) =
    (x+z)^k   \big[a\xz\, F_z(u,v)\big]_y(w).
\end{equation*}
Since $k$ does not depend on $w$, we immediately have (\ref{AA}),
as desired (cf. this argument with the last part of the proof of Theorem 5.22 in
\cite{Li3}).
\end{proof}

\

Using Proposition \ref{F}, we have $F\in
V\hbox{-}\,Bilinear(M,N;Vhom(W,M\oV (N\oV W)))$ given by
$[F_z(u,v)]_y(w)=u\ozy(v\oy w)$, and the natural vertex bilinear
map
$$
G\in V\hbox{-}\,Bilinear(Vhom(W,M\oV (N\oV W)),W;M\oV (N\oV W))
$$
given
by $G_y(q,w)=q_{\,y}(w)$. Now, applying Proposition \ref{3.2} to these
vertex bilinear maps, we obtain:

\begin{proposition}\label{3.10}
There exists a unique $V$-$\,$homomorphism $f:(M\oV N)\oV W\to M\oV
(N\oV W)$ such that
\begin{equation*}
f\big((u\oz v)\oy w\big)=u\ozy (v\oy w)
\end{equation*}
for $u\in M,v\in N$ and $w\in W$.
\end{proposition}

\

For any $v\in N$ and $w\in W$, we define a linear map
\begin{equation*}
    \phi(v,w\, ;z,y)\,:\, M\to \Big[(M\oV N)\oV W\Big]((z))((y))
\end{equation*}
given by
\begin{equation*}
\phi(v,w\,;z,y)(u):=(u\ozmy v)\oy w.
\end{equation*}

\vskip .1cm

\noi It follows from the assumption (\ref{cond2}), that this expression
 exists and if
we take
\begin{equation*}
\phi(v,w\,;z,y)=\sum_{n\in \ZZ}\, \phi_n(v,w\,;z)\, y^{-n-1}
\end{equation*}
then $\phi_k(v,w\,;z)\equiv 0$ for $k$ sufficiently large and
\begin{equation*}
\phi_n(v,w\, ;z):M\to \Big[(M\oV N)\oV W\Big]((z)).
\end{equation*}

\

\begin{proposition}
For any $v\in N, w\in W$ and $n\in \ZZ$, we have
\begin{equation*}
\phi_n(v,w\,;z)\in Vhom^r \big(M,(M\oV N)\oV W\big).
\end{equation*}
\end{proposition}

\

\begin{proof}
For $v\in N, w\in W$ and $n\in \ZZ$, we have
\begin{align*}
\phi(v,w\,;z,y)(du) & =\big(du\ozmy v\big)\oy w=\Big[\frac{d}{dx}
\big((u\ox v)\oy w\big)\Big]_{x=z-y}=\frac{d}{dz} \big((u\ozmy
v)\oy w\big)\\
& =\frac{d}{dz} \, \phi(v,w\,;z,y)(u).
\end{align*}
Then, $\phi_n(v,w\,;z)$ satisfies (a) in the definition of $Vhom^r$.
Let $N\in\NN$ be such that  $\phi_m(v,w\,;z)=0$ for all $m\geq N$.
Now, for any $a\in V$ and $n\in \ZZ$, we take $k\in \NN$ such that
$k+n\geq N$, and
\begin{align*}
    x^k\ a\x(u\oz v) & = x^k\ (a\xmz u)\oz v \qquad\qquad \,  \ \
    \hbox{ for all
    } u\in M,\\
    x^k\ a\x\big((u\oz v)\oy w \big) & = x^k\ \big( a\xmy (u\oz
    v)\big) \oy w \qquad \hbox{ for all } u\in M \hbox{ and } v\in
    N.
\end{align*}
Then, we obtain
\begin{align*}
     x^{3k}\ a\x  \big( \phi_n(v,w\,;z)(u)\big) &=
     \hbox{Res}_{\, y}\, y^nx^{3k} \, a\x \big(
     \phi(v,w\,;z,y)(u)\big)\\
     & = \hbox{Res}_{\, y}\, \sum_{i=0}^{2k}\binom{2k}{i} \,
     (x-y)^{2k-i} y^{n+i}x^k\, a\x \big(
     \phi(v,w\,;z,y)(u)\big)\\
     & = \hbox{Res}_{\, y}\, \sum_{i=0}^{k}\binom{2k}{i} \,
     (x-y)^{2k-i} y^{n+i}x^k\, a\x \big(
     \phi(v,w\,;z,y)(u)\big)\\
     & = \hbox{Res}_{\, y}\, \sum_{i=0}^{k}\binom{2k}{i} \,
     (x-y)^{2k-i} y^{n+i}x^k\,  \big( a\xmy(u\ozmy v)\big)\oy w\\
     & = \hbox{Res}_{\, y}\, \sum_{i=0}^{k}\binom{2k}{i} \,
     (x-y)^{2k-i} y^{n+i}x^k\,  \big( (a\xmz u)\ozmy v\big)\oy w\\
     & = \hbox{Res}_{\, y}\, \sum_{i=0}^{2k}\binom{2k}{i} \,
     (x-y)^{2k-i} y^{n+i}x^k\,  \big( (a\xmz u)\ozmy v\big)\oy w\\
     & = \hbox{Res}_{\, y}\,
     x^{3k} y^{n}\,  \big( (a\xmz u)\ozmy v\big)\oy w \\
     & =
     \hbox{Res}_{\, y}\, y^nx^{3k} \,  \big(
     \phi(v,w\,;z,y)(a\xmz u)\big)\\
     & =x^{3k}\    \phi_n(v,w\,;z)(a\xmz u) \qquad\qquad
     \hbox{ for all } u\in M.
\end{align*}
Then, $\phi_n(v,w\,;z)\in Vhom^r(M,(M\oV N)\oV W)$ for all $v\in N,
w\in W$ and $n\in \ZZ$.
\end{proof}

Now, we define a $\kk$-bilinear map
\begin{equation*}
H_y:N\times W\to Vhom^r(M,(M\oV N)\oV W)((y))
\end{equation*}
given by
\begin{equation*}
\Big[H_y(v,w)\Big]_z(u):=\phi(v,w\,;z,y)(u)=(u \ozmy v) \oy w
\end{equation*}
for any $v\in N$, $w\in W$ and $u\in M$.

\

\begin{proposition}\label{pppp}
 $H_y$ is a vertex bilinear map of type $(N, W; Vhom^r(M,(M\oV
N)\oV W))$.
\end{proposition}

\begin{proof}
For any $u\in M,v\in N$ and $w\in W$, we have
\begin{align*}
\Big[H_y(dv,w)\Big]_z(u) & = (u\ozmy dv)\oy
w=\Big[\Big(d-\frac{d}{dx}\Big)(u\ox v)\Big]\oy w_{\,|_{x=z-y}}\\
    & =\Big(\frac{d}{dy}\big[(u\ox v)\oy
    w\big]\Big)_{|_{x=z-y}}-\frac{d}{dx}\big[(u\ox v)\oy
    w\big]_{|_{x=z-y}}\\
    & =\frac{d}{dy}\big((u\ozmy v)\oy
    w\big)=\frac{d}{dy} \Big[H_y(v,w)\Big]_z(u),
\end{align*}
and
\begin{align*}
\Big[\Big(d-\frac{d}{dy} \Big)\, H_y(v,w)\Big]_z(u) & =
\Big(d-\frac{d}{dz}-\frac{d}{dy} \Big)\, \Big[H_y(v,w)\Big]_z(u) =
\Big(d-\frac{d}{dz}-\frac{d}{dy} \Big)\, \big((u\ozmy v)\oy
w\big)\\
& = \Big[\Big(d-\frac{d}{dy} \Big)\, \big((u\ox v)\oy
w\big)\Big]_{x=z-y}=(u\ozmy v)\oy dw=\Big[H_y(v,dw)\Big]_z(u).\\
\end{align*}
Then $H_y$ satisfies the properties with respect to $d$. Now, we
prove that $H_y$ satisfies locality or (\ref{BB}). For $a\in
V,u\in M$ and $v\in N$, we take $k$ and $n$ positive integers such
that
\begin{equation*}
x^k\, a\x u\in M[[x]]\ \ \ \   \hbox{ and }\ \ \ \  x^n \, a\x
v\in N[[x]].
\end{equation*}
Given $w\in W$ and using the action on $Vhom^r$ and the action on
the tensor product, we have
\begin{align*}
    (x-y)^n(x-z)^k & \Big(\big[a\x H_y(v,w)\big]_z(u) -
    \big[ H_y(v,a\x w)\big]_z(u)\Big)=\\
    &=(x-y)^n(x-z)^k \ \Big(a\x\big((u\ozmy v)\oy w\big) - (u\ozmy
    v)\oy a\x w\Big)\\
    &=(x-y)^n(x-z)^k \ \hbox{Res}_{\, t}\ \delta(y+t,x)\big(a\t
    (u\ozmy v)\big)\oy w\\
    &=(x-y)^n \ \hbox{Res}_{\, t}\ \delta(y+t,x)\, (t-(z-y))^k\ \big(a\t
    (u\ozmy v)\big)\oy w\\
    &=(x-y)^n(x-z)^k \ \hbox{Res}_{\, t}\ \delta(y+t,x) \
    \big(u\ozmy a\t v\big)\oy w\\
    &=(x-z)^k \ \hbox{Res}_{\, t}\
    \Big(\delta(x-y,t)-\delta(-y+x,t)\Big) \ (x-y)^n\
    \big(u\ozmy a\t v\big)\oy w =0
\end{align*}
Therefore
\begin{equation}\label{qqqq}
    (x-y)^n(x-z)^k \big[a\x H_y(v,w)\big]_z(u) = (x-y)^n(x-z)^k
    \big[ H_y(v,a\x w)\big]_z(u).
\end{equation}
Observe that we are allowed to multiply both sides of (\ref{qqqq})
by $(-z+x)^{-k}$, and we obtain
\begin{equation*}
(x-y)^n \big[a\x H_y(v,w)\big]_z(u) = (x-y)^n
    \big[ H_y(v,a\x w)\big]_z(u).
\end{equation*}
Since $n$ does not depend on $u$, we have proved locality.

Now, we prove associativity or (\ref{AA}). Let $k$ be a positive
integer (depending only on $a\in V$ and $w\in W$) such that
\begin{equation}\label{eqqq}
    (x+y)^k\ a\xxy (q\oy w)= (x+y)^k\ (a\x q)\oy w
\end{equation}
for any $q\in M\oV N$. Then, using the action on $Vhom^r$,
(\ref{eqqq}) and the action on the tensor product, we obtain
\begin{align*}
     (x+y)^k \ &\Big[a\xxy H_y(v,w)\Big]_z(u)=\\
     &=(x+y)^k \
     a\xxy\big((u\ozmy v)\oy w\big)-\hbox{Res}_{\, t}\
     \delta(z+t,x+y) \, (x+y)^k \big((a\t u)\ozmy v\big)\oy w\\
&=(x+y)^k \
     \big(a\x (u\ozmy v)\big)\oy w-\hbox{Res}_{\, t}\
     \delta(z+t,x+y) \, (x+y)^k \big((a\t u)\ozmy v\big)\oy w\\
&=(x+y)^k \
      \big(u\ozmy a\x v\big)\oy w +\hbox{Res}_{\, t}\
     \delta(x-t,z-y) \, (x+y)^k \big((a\t u)\ozmy v\big)\oy w\\
&\hskip 4.57cm -\hbox{Res}_{\, t}\
     \delta(z+t,x+y) \, (x+y)^k \big((a\t u)\ozmy v\big)\oy w\\
&=(x+y)^k \
      \big(u\ozmy a\x v\big)\oy w\\
&=(x+y)^k\Big[H_y(a\x v,w)\Big]_z(u) \qquad \quad \hbox{ for all
}u\in M.
\end{align*}
This proves the associativity or (\ref{AA}). Therefore $H$ is a
vertex bilinear map.
\end{proof}

Using Proposition \ref{pppp}, we have $H\in
V\hbox{-}\,Bilinear(N,W;Vhom^r(M,(M\oV N)\oV W))$ given by
\begin{equation*}
\Big[H_y(u,v)\Big]_z(u)=(u\ozmy v)\oy w,
\end{equation*}

\vskip .2cm

\noi and the natural vertex bilinear map $I$ of type
$(M,Vhom^r(M,(M\oV N)\oV W); (M\oV N)\oV W)$ given by
$I_y(u,h)=h_y(u)$. Now, applying Proposition \ref{3.1} to these
vertex bilinear maps, we obtain:

\begin{proposition}\label{3.14}
There exists a unique $V$-$\, $homomorphism $g:M\oV (N\oV W)\to (M\oV
N)\oV W$ such that
\begin{equation*}
g\big(u\oz(v\oy w)\big)=(u\ozmy v)\oy w
\end{equation*}

\vskip .2cm

\noi for all $u\in M,v\in N$ and $w\in W$.
\end{proposition}

\noi By Proposition \ref{3.10} and Proposition \ref{3.14}, we
have:

\begin{theorem}
Let $V$ be a vertex algebra and $M,N$ and $W$ be three
$V$-$\,$modules. Assume that the tensor products $M\oV N, N\oV W,
(M\oV N)\oV W$ and $M\oV (N\oV W)$ exist, and assume that

\vskip .2cm

\noi (1) for any $u\in M$, $v\in N$ and $w\in W$

\begin{equation*}
 u\oxy (v\oy w)\in \Big[M\oV (N\oV W)\Big]((y))((x))
\end{equation*}

\vskip .3cm

\noi where the number of negative powers of $x$ depends only on
$u$ and $v$,
and

\vskip .2cm

\noi (2) for any $u\in M$, $v\in N$ and $w\in W$

\begin{equation*}
 (u\oxmy v)\oy w\in \Big[(M\oV N)\oV W\Big]((x))((y))
\end{equation*}

\vskip .3cm

\noi where the number of negative powers of $y$ depends only on
$v$ and $w$. Then the
$V$-$\,$homomorphisms $f$ and $g$ in Proposition \ref{3.10} and
Proposition \ref{3.14}, are $V$-$\,$isomorphisms between $M\oV (N\oV
W)$ and $(M\oV N)\oV W$.
\end{theorem}

\vskip .3cm

%%%%%%%%%%%%%%%%%%%%%%%%%%
\section{Exactness of functors}\lbb{ult}
%%%%%%%%%%%%%%%%%%%%%%%%%%

\

We will introduce the notions of vertex projective, vertex
injective and vertex flat $V$-$\,$modules. Roughly speaking,
homological algebra is concerned with the question of how much
modules differ from being projective, injective and flat.

We do not try to develop homological algebra theory for vertex
algebras in this work. This is just an introduction or motivation
to the reader to try to develop it from the point of view of a
vertex algebra as a generalization of an associative commutative
algebra with unit.

Now, we examine the
behavior of $Vhom$ with respect to exact sequences.

\begin{theorem} \label{exa}
{\rm (Left exactness)} (a) If a sequence of $V$-$\,$modules
$$\begin{CD}
0@>>> N\s ' @>\varphi >> N @> \psi >> N\s ''
\end{CD}$$

\vskip .3cm

\noi is exact, then for any $V$-$\,$module $M$ we have that
$$\begin{CD}
0@>>> Vhom(M,N\s ') @>\varphi_* >> Vhom(M,N) @> \psi_* >>
Vhom(M,N\s '')
\end{CD}$$

\vskip .3cm

\noi is an exact sequence of $V$-$\,$modules.

\vskip .3cm

\noi (b) If a sequence of $V$-$\,$modules
$$\begin{CD}
 M\s ' @>\varphi >> M @> \psi >> M\s '' @>>> 0
\end{CD}$$

\vskip .3cm

\noi is exact, then for any $V$-$\,$module $N$ we have that
$$\begin{CD}
0@>>> Vhom(M\s '',N) @>\psi^* >> Vhom(M,N) @> \varphi^* >>
Vhom(M\s ',N)
\end{CD}$$

\vskip .3cm

\noi is an exact sequence of $V$-$\,$modules.
\end{theorem}

\vskip .3cm

\begin{proof} (a)
If $\ \ 0\xrightarrow{ \ \ \ \ } N\s '\xrightarrow{ \ \ \varphi \
\ } N
  \xrightarrow{ \ \ \psi \ \ } N\s ''$ is exact, then:

\vskip .2cm

(i) Ker $\varphi_*=0$: if $f\in$ Ker $\varphi_*$ with $f:M\to N\s
'((z))$, then $\varphi_*(f)=0$, that is $\varphi(f_z(u))=0$ for
all $u\in M$. Since $\varphi$ is injective, we have $f=0$.
Therefore Ker $\varphi_*=0$.

\vskip .2cm

(ii) Im $\varphi_*\subseteq$ Ker $\psi_*$: Since Im $\varphi =$ Ker
$\psi$ by exactness, we have $\psi\varphi=0$ and hence
$\psi_*\varphi_*=( \psi \varphi)_*=0$. Therefore Im
$\varphi_*\subseteq$ Ker $\psi_*$.

\vskip .2cm

(iii) Ker $\psi_*\subseteq$ Im $\varphi_*$: if $g\in $ Ker
$\psi_*$, then $\psi g=0$ and Res$\s_z \, z^m g_z(u)\in $ Ker
$\psi=$ Im $\varphi$ for all $u\in M$, $m\in \ZZ$. Since $\varphi$
is a monomorphism, $\varphi:N\s '\to $ Im $\varphi$ is an
isomorphism. Define $h$ by $h=\varphi^{-1} g$, where
$\varphi^{-1}$ is defined on Im $\varphi$, then $h\in Vhom(M,N\s
')$ and $g=\varphi h=\varphi_*(h)$. Therefore Ker
$\psi_*\subseteq$ Im $\varphi_*$.

\vskip .4cm

\noi (b) If $ M\s '\xrightarrow{ \ \ \varphi \ \ } M
  \xrightarrow{ \ \ \psi \ \ } M\s '' \xrightarrow{ \ \ \ \ }0$ is exact, then:

\vskip .2cm

(i) Ker $\psi^*=0$: if $f\in $ Ker $\psi^*$ with $f\in
Vhom(M\s'',N)$, then $\psi^*(f)=0$, that is $f_z(\psi(u))=0$ for
all $u\in M$. Thus, $f_z(u'')=0$ for all $u''\in $ Im $\psi$.
Since $\psi$ is surjective, we have $f=0$. Therefore Ker
$\psi^*=0$.

\vskip .2cm

(ii) Im $\psi^*\subseteq$ Ker $\varphi^*$: if $g\in Vhom(M\,
'',N)$, then $\varphi^*(\psi^*(g))=(\psi\varphi)^*(g)=0$ because
exactness of the original sequence.

\vskip .2cm

(iii) Ker $\varphi^*\subseteq$ Im $\psi^*$: if $g\in$  Ker
$\varphi^*$, then $g:M\to N((z))$ and $g_z(\varphi(u'))=0$ for all
$u'\in M\s '$. If $u''\in M\s ''$, then $u''=\psi(u)$ for some $u
\in M$, because $\psi$ is surjective. Define $f:M\s '' \to N((z))$
by $f_z(u'')=g_z(u)$ if $u''=\psi(u)$. Note that $f$ is well
defined: if $\psi(u)=\psi(\tilde u)$, then $u-\tilde u\in$ Ker
$\psi=$ Im $\varphi$, so that $u-\tilde u=\varphi(u')$ for some
$u'\in M\s '$. Hence, $g_z(u)-g_z(\tilde u)=g_z(u-\tilde
u)=g_z(\varphi(u'))=0$ because $g_z\varphi=0$. The reader may
check that $f$ is a vertex homomorphism. Finally, $\psi^*(f)=f\psi
=g$, because if $u''=\psi(u)$, then $g_z(u)=f_z(u'')=f(\psi(u))$.
Therefore $g\in $ Im $\psi^*$.
\end{proof}

\

We have proved that $Vhom(M,\square)$ and $Vhom(\,\square, N)$ are
left exact functors. It is {\it not} true in general that
$Vhom(M,\square)$ and $Vhom(\,\square, N)$ are (right) exact
functors. Because, if we consider $V$ the vertex algebra given by
an associative commutative algebra with unit and $d=0$, then
everything collapses to the classical case, that is
$Vhom(M,N)=Vhom^r(M,N)=Hom_V(M,N)$ and we know that $Hom_V$ is not
exact in both variables. This is the motivation for the
introduction of the notions of vertex projective and vertex
injective $V$-$\,$modules.

\

\begin{definition}
A $V$-$\,$module $P$ is said to be {\it vertex projective} if for any
surjective $V$-$\,$homo- morphism $\psi:N\to N\s ''$ and any $f\in
Vhom(P,N\s '')$, there exists a lifting $g$, that is, there exists
$g\in Vhom(P,N)$ making the following diagram commute:

\begin{displaymath}
\xymatrix{
& P \ar@{.>}[dl]_{g} \ar[d]^f \\
N((z)) \ar[r]_\psi & N''((z)) \ar[r] & 0}
\end{displaymath}
\end{definition}

\vskip .3cm

\begin{proposition}
A $V$-$\,$module $P$ is vertex projective if and only if
$Vhom(P,\square)$ is an exact functor.
\end{proposition}

\begin{proof}
If $P$ is vertex projective and $\psi:N\to N''$ is a surjective
$V$-$\,$homomorphism, then given $f\in Vhom(P,N'')$ there exists a
lifting $g:P\to N((z))$ with $f=\psi g=\psi_*(g)\in \,$Im
$\psi_*$, and so $\psi_*$ is surjective. Hence, by Theorem
\ref{exa} we have that $Vhom(P,\square )$ is an exact functor. The
converse follows by reversing the previous argument.
\end{proof}

\vskip .3cm

\begin{definition}
A $V$-$\,$module $E$ is  {\it vertex injective} if for any injective
$V$-$\,$homomorphism $\varphi:M\s'\to M$ and any $f\in Vhom(M\s ',E)$,
there exists $g\in Vhom(M,E)$ making the following diagram
commute:

\begin{displaymath}
\xymatrix{
0   \ar[r]
&M' \ar[d]_f \ar[r]^\varphi
&M \ar@{.>}[ld]^g
\\
&E((z))}
\end{displaymath}

\end{definition}

In words, every vertex homomorphism from a submodule into $E$ can
always be extended to a vertex homomorphism from the big module
into $E$.

\begin{proposition}
A $V$-$\,$module $E$ is vertex injective if and only if $Vhom(\square,
E)$ is an exact functor.
\end{proposition}

\begin{proof}
By Theorem \ref{exa}, $Vhom(\square , E)$ is a left exact
contravariant functor, therefore the thrust of the proposition is
that the induced map $\varphi^*$ is surjective whenever $\varphi$
is injective, and this is exactly the definition of vertex
injective module.
\end{proof}

\

Now, we examine the behavior of tensor products with respect to
exact sequences:

\begin{theorem}
{\rm (Right exactness)} If $ N\s '\xrightarrow{ \ \ f \ \ } N
  \xrightarrow{ \ \ g \ \ } N\s '' \xrightarrow{ \ \ \ \ }0$ is an
  exact sequence of $V$-$\,$modules and $M$ is a $V$-$\,$module, then
$$
M\oV N\s '\xrightarrow{ \ 1_M\otimes f  \ } M\oV N
  \xrightarrow{ \ 1_M\otimes g  \ } M\oV N\s '' \xrightarrow{ \ \ \ \ }0
$$
is an  exact sequence of $V$-$\,$modules.
\end{theorem}

\begin{proof}
There are three things to check.

\vskip .3cm

\noi (i) $1_M\otimes g$ is surjective: given $v''\in N\s ''$ and
$u\in M$, there exists $v\in N$ such that $g(v)=v''$, and hence
\begin{equation*}
u\oz v''=u\oz g(v)=(1_M\otimes g)(u\oz v).
\end{equation*}
Since elements of type $u\oz v''$ (in fact, the coefficients in
$z$) generate $M\oV N\s ''$, we conclude that $1_M\otimes g$ is
surjective.

\vskip .3cm

\noi (ii) Im $(1_M\otimes f)\subseteq $ Ker $(1_M\otimes g)$: it
suffices to prove that the composite is 0, but using Corollary
\ref{coco} we have
\begin{equation*}
(1_M\otimes f)(1_M\otimes g)=1_M\otimes fg =1_M\otimes 0=0.
\end{equation*}

\vskip .3cm

\noi (iii) Ker $(1_M\otimes g)\subseteq$ Im $(1_M\otimes f)$: let
$E=$Im $(1_M\otimes f)$. By part (ii), $E\subseteq $ Ker
$(1_M\otimes g)$, and so $1_M\otimes g$ induces a map $\hat g
:(M\oV N)/E\to M\oV N\s ''$ with $\hat g (u\oz v+E)=u\oz g(v)$,
where $u\in M$ and $v\in N$. Now, if $\pi : M\oV N\to (M\oV N)/E$
is the canonical epimorphism, then $\hat g \pi =1_M\otimes g$. We
shall define a map $h:M\oV N\s '' \to (M\oV N)/E$ such that $h\,
\hat g =id$. This obviously will imply that $\hat g$ is injective,
and hence
\begin{equation*}
\hbox{Ker }(1_M\otimes g)=\hbox{Ker }(\hat g \pi)=\hbox{Ker }
(\pi)=\hbox{Im } (1_M\otimes f)
\end{equation*}
and we are done. First, we show that the map
\begin{equation*}
\tilde h_z:M\times N\s '' \longrightarrow \big[(M\oV N)/E\big ]\,
((z))
\end{equation*}
given by $\tilde h_z(u,v'')=\pi(u\oz v)$, where $g(v)=v''$, is
well defined, i.e. is independent of the choice of $v$. Note that
there is at least on such $v$ since $g$ is surjective. If
$g(v_1)=g(v_2)=v''$, then $g(v_1-v_2)=0$, and by hypothesis,
$v_1-v_2=f(v')$ for some $v'\in N\s '$. Then each coefficient in
$z$ of
\begin{equation*}
u\oz v_1-u\oz v_2=u\oz (v_1-v_2)=u\oz f(v')
\end{equation*}
is in Im $(1_M\otimes f)=E$, and this proves that our map is well
defined. It is obviously a vertex bilinear map, and by the
universal property, there is a unique $V$-$\,$homomorphism
\begin{equation*}
h:M\oV N\s '' \longrightarrow (M\oV N)/E
\end{equation*}
such that $h(u\oz v'')=\pi (u\oz v)$, where $g(v)=v''$. It is
clear that
\begin{equation*}
h(\hat g(u\oz v +E))=h(u\oz g(v))=u\oz v +E,
\end{equation*}
hence $h\hat g$ is the identity map. So, we conclude that $\hat g$
is injective, as was to be shown.
\end{proof}

\

The next type of module arises from tensor products in the same
way that vertex projective and vertex injective modules arise from
$Vhom$. If $V$ is a vertex algebra given by an associative
commutative algebra with unit, the tensor product coincides with
the usual tensor product of modules and we know that the
associated functor is not left exact.

\begin{definition}
A $V$-$\,$module $M$ is {\it vertex flat} if $M\oV \square$ is an
exact functor.
\end{definition}

Because the functors $M\oV \square$ are right exact, we see that a
$V$-$\,$module $M$ is vertex flat if and only if, whenever $i\in
Hom_V(N\s ',N)$ is an injection, then $1_M\otimes i:M\oV N\s ' \to
M\oV N$ is also an injection.

\

\

\subsection*{Acknowledgements}
The  author was supported  by a grant by Conicet, Consejo Nacional
de Investigaciones Cient\'ificas y T\'ecnicas (Argentina). The
author would like to thank  C. Boyallian, M. V. Postiguillo and C.
Bortni for their constant help, care and support to finish this work.
The author would like to thank F. Orosz for her help, patience,
kindliness and constant support throughout the first part of
this work.

%%%%%%%%%%%%%%%%%%%%%%%%%%%%%%%%%%%%%%%%%%%%%%%%%%%%%%%%%%%%%%%%%%%%%%%%%
\bibliographystyle{amsalpha}

%%%%%%%%%%%%%%%%%%%%%%%%%%%%%%%%%%%%%%%%%%%%%%%%%%%%%%%%%%%%%%%%%%%%%%%%%

%%%%%%%%%%%%%%%%%%%%%%%%%%%%%%%%%%%%%%%%%%%%%%%%%%%%%%%%%%%%%%%%%%%%%%%%%
\end{document}